\documentclass[12pt,twoside]{article}
\usepackage{geometry}
\usepackage{amsmath,amsfonts,amssymb,amsthm,a4}
\usepackage{url} 
\usepackage{etoolbox} 
\usepackage{epsfig,color} 

\newtheorem{theo}{Theorem}
\newtheorem{prop}{Proposition}[section]
\newtheorem{coro}[prop]{Corollary}
\newtheorem{lemma}[prop]{Lemma}
\newtheorem{conj}[prop]{Conjecture}

\theoremstyle{definition}
\newtheorem{defin}[prop]{Definition}
\newtheorem{example}[prop]{Example}
\AtEndEnvironment{example}{\null\hfill$\diamond$}%
\newtheorem{remark}[prop]{Remark}
\AtEndEnvironment{remark}{\null\hfill$\diamond$}%

\newcommand{\North}{\operatorname{North}}
\newcommand{\South}{\operatorname{South}}
\newcommand{\East}{\operatorname{East}}
\newcommand{\West}{\operatorname{West}}

\newcommand{\NE}{\operatorname{NE}}
\newcommand{\SE}{\operatorname{SE}}
\newcommand{\NW}{\operatorname{NW}}
\newcommand{\SW}{\operatorname{SW}}

\newcommand{\bt}{{\mathbf t}}

\newcommand{\sign}{\operatorname{sign}}

\newcommand{\NN}{{\mathbb{N}}}
\newcommand{\ZZ}{{\mathbb{Z}}}

\newcommand{\RR}{{\mathbb{R}}}

\newcommand{\cL}{{\cal L}}

\newcommand{\cD}{{\cal D}}

\newcommand{\cR}{{\cal R}}

\newcommand{\cM}{{\cal M}}

\newcommand{\cX}{{\cal X}}

\newcommand{\Tw}{\operatorname{Tw}}
\newcommand{\TTw}{\operatorname{TTw}}

\begin{document}
\title{Slab tilings, flips and the triple twist}
\author{George~L.~D.~Alencar \and Nicolau~C.~Saldanha \and
Arthur~M.~M.~Vieira}

\maketitle


\begin{abstract}
A \textit{domino} is a $2\times 1\times 1$ parallelepiped
formed by the union of two unit cubes
and a \textit{slab} is a $2\times 2\times 1$ parallelepiped
formed by the union of four unit cubes.
We are interested in tiling
regions formed by the finite union of unit cubes.
Domino tilings have been studied before;
here we investigate \textit{slab tilings}.
As for domino tilings, a flip in a slab tiling is a local move:
two neighboring parallel slabs
are removed and placed back in a different position.

Inspired by the twist for domino tilings,
we construct a flip invariant for slab tilings:
the \textit{triple twist}, assuming values in $\ZZ^3$.
We show that if the region is a large box
then the triple twist assumes
a large number of possible values,
roughly proportional to the fourth power of the volume.
We also give examples of smaller regions for which
the set of tilings is connected under flips,
so that the triple twist assumes only one value.%
\end{abstract}



\section{Introduction}
\label{section:introduction}

In this paper we consider tilings of cylinders,
of which rectangular boxes are a special case.
A cylinder $\cR$ is a cubiculated region in $\mathbb{R}^3$
of the form $\cR = \cD\times [0,N]$,
where $\cD$ is a disk, a quadriculated connected and
simply connected region on the $z=0$ plane.
Unit cubes are of the form $[x,x+1]\times[y,y+1]\times[z,z+1]$
where $(x,y,z) \in \ZZ^3$.

We give a brief summary of the directly relevant results
about 2d and 3d domino tilings.
We choose to depict a three-dimensional domino tiling by floors,
numbered from left to right.
The $n$-th floor from left to right
corresponds to $z \in [n,n+1]$, $n \in \ZZ$.
A domino which is contained in only one floor
(named \textit{horizontal} domino)
is represented as a $2\times 1$ light grey rectangle.
A domino which crosses two floors (named \textit{vertical})
is represented as two $1\times 1$ squares
which have the same projection on the first floor,
with the square on the lower floor (left) being painted dark grey
and the square on the upper floor (right) being painted white.
Figure~\ref{fig:cobtw} shows a domino tiling of the $3\times 3\times 2$ box.
On the left, the floor $z \in [0,1]$; on the right, the floor $z \in [1,2]$.
The $x$ and $y$ axes are usually drawn as in this example.
For convenience, we may at times speak of a region
using vocabulary of the floor diagram for the region.
That is, we talk of the corresponding square of the floor diagram
when referring to a certain cube.

\begin{figure}[ht]
\centering
 \def\svgwidth{08cm}
\begingroup%
  \makeatletter%
  \providecommand\color[2][]{%
    \errmessage{(Inkscape) Color is used for the text in Inkscape, but the package 'color.sty' is not loaded}%
    \renewcommand\color[2][]{}%
  }%
  \providecommand\transparent[1]{%
    \errmessage{(Inkscape) Transparency is used (non-zero) for the text in Inkscape, but the package 'transparent.sty' is not loaded}%
    \renewcommand\transparent[1]{}%
  }%
  \providecommand\rotatebox[2]{#2}%
  \newcommand*\fsize{\dimexpr\f@size pt\relax}%
  \newcommand*\lineheight[1]{\fontsize{\fsize}{#1\fsize}\selectfont}%
  \ifx\svgwidth\undefined%
    \setlength{\unitlength}{300bp}%
    \ifx\svgscale\undefined%
      \relax%
    \else%
      \setlength{\unitlength}{\unitlength * \real{\svgscale}}%
    \fi%
  \else%
    \setlength{\unitlength}{\svgwidth}%
  \fi%
  \global\let\svgwidth\undefined%
  \global\let\svgscale\undefined%
  \makeatother%
  \begin{picture}(1,0.33666667)%
    \lineheight{1}%
    \setlength\tabcolsep{0pt}%
    \put(0,0){\includegraphics[width=\unitlength,page=1]{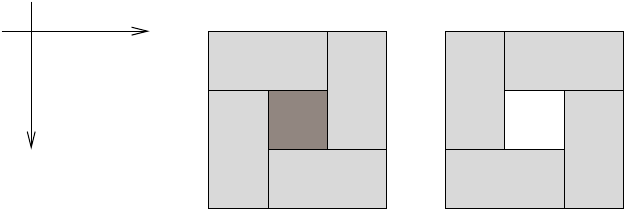}}%
    \put(0.18259849,0.2299187){\color[rgb]{0,0,0}\makebox(0,0)[lt]{\lineheight{1.25}\smash{\begin{tabular}[t]{l}$y$\end{tabular}}}}%
    \put(0.08785177,0.0593746){\color[rgb]{0,0,0}\makebox(0,0)[lt]{\lineheight{1.25}\smash{\begin{tabular}[t]{l}$x$\end{tabular}}}}%
  \end{picture}%
\endgroup%

\caption{Domino tiling of a $3\times 3\times 2$ box}
\label{fig:cobtw}
\end{figure}

We are interested in local moves.
The most natural local move is a \textit{flip}:
remove two parallel dominoes and place them back in
the only other possible position.
We provide a graph theoretical formulation of this scenario.
Given a region $\cD$, construct a new graph,
where vertices are domino tilings $\bt$ of $\cD$
and edges join pairs $\bt_0, \bt_1$ of tilings which differ by a flip.
If $\cD \subset \RR^2$ is a quadriculated disk,
i.e., a bounded 2d contractible region,
then this graph is connected.
Equivalently,
any two tilings of $\cD$ can be joined
by a finite sequence of flips \cite{thurston1990}.
Applied to other regions (possibly 3d),
this construction allows us to speak of
\textit{connected components} under flips.

For 3d domino tilings, this new graph is almost never connected.
As a first example, consider the box in Figure~\ref{fig:cobtw}
and the graph whose vertices are its tilings.
This box admits $229$ tilings
(see \cite{regulardisk},
which includes larger examples).
The tiling in Figure~\ref{fig:cobtw} admits no flips
(it corresponds to an isolated vertex in the graph).
The mirror image is likewise isolated.
The other $227$ tilings
form a single connected component under flips.
It might be conjectured after this example
that \textit{almost} all tilings of a 3d box
can be joined by a finite number of flips:
as we shall see, this is also not correct.

Given a contractible cubiculated region $\cR \subset \RR^3$,
there exists a map $\Tw$ (twist)
from the set of domino tilings of $\cR$ to $\ZZ$.
There are several ways to construct the function $\Tw$:
see \cite{FKMS,KS} for more general constructions.
For completeness,
we include the construction from \cite{segundoartigo}
in Section~\ref{section:dominotwist}:
that construction is valid if $\cR$ is a cylinder.
This map has several important properties.
The main property is that twist is an invariant under domino flips,
so that connected components under flips
are contained in level sets of $\Tw$.
For the $3 \times 3\times 2$ box,
the tiling in Figure~\ref{fig:cobtw} has $\Tw = -1$;
its mirror image has $\Tw = +1$.
The other $227$ tilings all have $\Tw = 0$.

Twist also works well with symmetries of the region $\cR$.
For instance, with our construction,
if $\cR$ is a cylinder
then the twist is left unchanged after a rotation in $\RR^3$.
Also for cylinders,
the twist changes sign under reflections by any of the planes
$x=0$, $y=0$ or $z=0$.

For domino tilings,
there is another important local move besides the flip:
the trit (\cite{segundoartigo,FKMS}).
In order to perform a trit in a tiling $\bt$,
search for three dominoes 
whose union is a $2\times 2\times 2$ cube minus
two unit cubes in opposite vertices.
The trit is then performed by removing the three dominoes
and placing them back in the only other possible position.
A trit always changes the twist by adding $\pm 1$,
where the sign depends on orientation.
The tiling in Figure~\ref{fig:cobtw}
admits four possible trits,
all of them taking us to tilings of twist $0$.

There is no fixed finite set of local moves
connecting all domino tilings
which works for any contractible cubiculated region $\cR \subset \RR^3$.
It is conjectured that any two tilings of a box
can be connected by a finite sequence of flips and trits.
It is known for large classes of boxes
that this holds for \textit{almost all} tilings \cite{regulardisk}.
It is also known that any tiling of a box admits
at least one flip or trit \cite{HLT}.

It is conjectured that for large boxes
the twist follows a normal distribution.
This is known to be true for boxes of size $L\times M\times N$
where $L, M \ge 4$ are even and fixed and $N$ tends to infinity \cite{normal}.
In this case, it is also known that
for almost all pairs $(\bt_0, \bt_1)$ of tilings
with $\Tw(\bt_0) = \Tw(\bt_1)$ 
there exists a finite sequence of flips connecting $\bt_0$ and $\bt_1$.
This implies that connected components under flips
almost coincide with level sets of the function $\Tw$.
Equivalently, within a (large) level set of $\Tw$
there exists a giant component.


\bigskip

The aim of the present paper is to investigate
related questions for \textit{slab} tilings of 3d regions.
A slab is a $2\times 2\times 1$ cuboid.
Let us draw a small example in Figure~\ref{fig:cob}.
As for domino tilings, we depict a slab tiling by floors,
numbered from left to right.
A slab which is contained in only one floor (\textit{horizontal})
is represented as a $2\times 2$ light grey square.
A slab which crosses two floors (\textit{vertical})
is represented as two planar dominoes.
These two dominoes must have the same projection on the first floor.
The domino on the lower (left) floor is painted dark grey
and the domino on the upper (right) floor is painted white.

\begin{figure}[ht]
    \centering
 \includegraphics[scale=0.2]{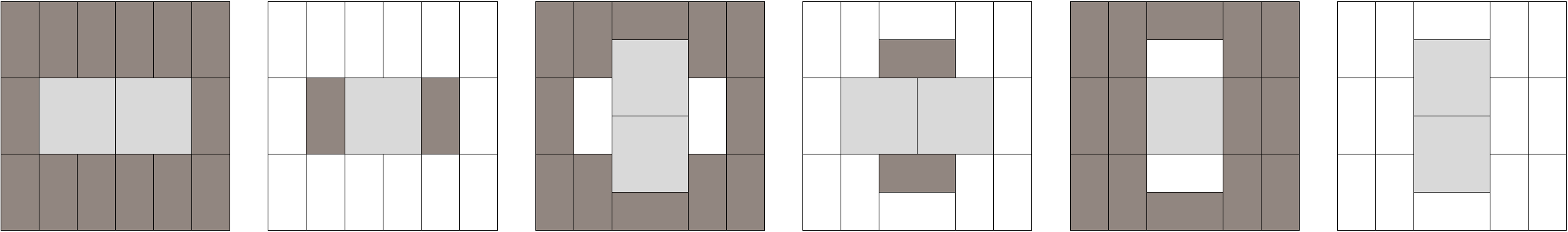}
    \caption{A slab tiling of a $6\times 6\times 6$ box}
    \label{fig:cob}
\end{figure}

We define \textit{flips} for slabs in the same way as for dominoes:
we take two slabs which have a $2\times 2$ face in common,
remove them from the region,
and then place them back in one of the two other positions. 

\goodbreak

A natural question is to investigate connected components via flips
of the set of slab tilings of a given region.
The existence of the twist as an invariant under flips
is crucial for the study of similar questions for domino tilings:
does a similar invariant exist for slab tilings?
As we shall see, the answer is yes,
but there are many significant differences.
The new invariant is the \textit{triple twist} ($\TTw$),
which assumes values in $\ZZ^3$.
Its construction, which takes a large portion of the present paper,
shows that it is related to the twist (for dominoes)
and indeed invariant under flips.
The triple twist assumes a large number of values.



\bigbreak

\begin{theo}
\label{theo:tripletwist}
Consider a cubiculated box $\cR$ of size
$N \times N \times N$, $N \in 2\NN$.
If $N$ is sufficiently large then
there exists positive constants $C_{-} < C_{+}$ with the following properties.
If $\bt$ is a slab tiling of $\cR$ then $|\TTw(\bt)| < C_{+} N^4$.
If $t \in (2\ZZ)^3$ and $|t| < C_{-} N^4$ then
there exists a tiling $\bt$ of $\cR$ with $\TTw(\bt) = t$.
\end{theo}

\goodbreak

The previous theorem gives a lower bound for the number
of connected components under flips
of the set of slab tilings of a large box
(see Proposition~\ref{bound} for an explicit formula).
We easily obtain an upper bound for the set of values of $\TTw$,
but the number of connected components appears to be much larger
(and harder to estimate).
In some simple cases,
such as a $4\times 4\times N$ or an $M\times N\times 2$ box,
there is only one connected component.

The construction of the triple twist leads us to consider
\textit{mixed tilings},
consisting of horizontal slabs and vertical dominoes.
Given a slab tiling,
we construct a corresponding mixed tiling
(Section~\ref{section:mixed}).
If the initial region is a box,
this construction can be performed in any of the three canonical directions.
Given in turn a mixed tiling,
we use a four color pattern to obtain a corresponding domino tiling
(Section~\ref{section:coloring}):
we then compute its twist.
The triple twist is computed by performing
the above construction in all three directions.
As for the domino tilings,
these new versions of the twist are well behaved under symmetries.

As for domino tilings, there exist examples
of slab tilings which admit no flip
(see Figures~\ref{665} and \ref{885}).
There are similar examples of mixed tilings
which admit no flip (see Figure~\ref{fig:mixednotslab}).
We find it more surprising to observe that there exists
no fixed set of local moves which makes
the set of slab tilings connected under such moves.
More than that:
given a fixed set of local moves
(involving a bounded number of slabs)
there exist slab tilings which admit
none of these local moves
(see Lemma~\ref{lemma:moves}, Corollary~\ref{coro:nolocal}
and Remark~\ref{rem:nolocal}).


\smallskip

In Section~\ref{section:dominotwist}
we define and review basic properties of the twist for domino tilings.
We introduce mixed tilings in Section~\ref{section:mixed}.
In Section~\ref{section:coloring}
we consider a way of using four colors
to paint the unit cubes
in order to construct the twist invariants for mixed and slab tilings.
In Section~\ref{section:properties}
we discuss some properties of these invariants,
including the effect of symmetries.
Section~\ref{section:localmoves}
includes examples and results concerning
connectivity under local moves.
Theorem~\ref{theo:tripletwist} 
and related results are proved
in Section~\ref{section:valuestwist}.
Finally, in Section \ref{section:final}
we list some open problems.


\section{Construction of the twist}
\label{section:dominotwist}

This section contains one among several known constructions
of the \textit{twist}
for domino tilings of a cubiculated region $\cR \subset \RR^3$.
This is included for completeness; we follow \cite{segundoartigo}.
See \cite{FKMS, KS} for alternative constructions,
sometimes more general.

The \textit{effect} is
a function which takes two dominoes in a domino tiling $\bt$
and returns a value in $\{0,\pm 1\}$.
On $\bt$, we choose a horizontal domino $d_0$
such that the line through 
the centers of two cubes covered by $d_0$
is parallel to the base vector $\mathbf{e_2}$.
Let $d_1$ be a vertical domino in $\bt$.
We say $d_1$ \textit{affects} $d_0$ if any of the cubes of $d_1$
is on the same floor as $d_0$
and if the projections of the interiors of $d_0$ and $d_1$
over the plane $x=0$ have non-empty intersection.
Figure~\ref{fig:calctw} shows some examples for which there are effects.
If these conditions do not hold we say the effect of $d_1$ on $d_0$ is $0$.

In the case $d_1$ does affect $d_0$,
we define the effect as a sign based on some parameters.
The first is whether $d_1$ is on the left or on the right side of $d_0$.
The second is whether $d_1$ is above or below $d_0$.
The third is whether the distance between $d_1$ and $d_0$ is odd or even.
And the fourth and final is whether $d_1$'s lower half cube
is on or below $d_0$'s floor.
As a convention, the effect on the first example of Figure~\ref{fig:calctw}
is set as $+1$,
and a change on any of the parameters
above yields a change of sign on the effect.

\begin{figure}[ht]
    \centering
\def\svgwidth{12cm}
\begingroup%
  \makeatletter%
  \providecommand\color[2][]{%
    \errmessage{(Inkscape) Color is used for the text in Inkscape, but the package 'color.sty' is not loaded}%
    \renewcommand\color[2][]{}%
  }%
  \providecommand\transparent[1]{%
    \errmessage{(Inkscape) Transparency is used (non-zero) for the text in Inkscape, but the package 'transparent.sty' is not loaded}%
    \renewcommand\transparent[1]{}%
  }%
  \providecommand\rotatebox[2]{#2}%
  \newcommand*\fsize{\dimexpr\f@size pt\relax}%
  \newcommand*\lineheight[1]{\fontsize{\fsize}{#1\fsize}\selectfont}%
  \ifx\svgwidth\undefined%
    \setlength{\unitlength}{675bp}%
    \ifx\svgscale\undefined%
      \relax%
    \else%
      \setlength{\unitlength}{\unitlength * \real{\svgscale}}%
    \fi%
  \else%
    \setlength{\unitlength}{\svgwidth}%
  \fi%
  \global\let\svgwidth\undefined%
  \global\let\svgscale\undefined%
  \makeatother%
  \begin{picture}(1,0.25185185)%
    \lineheight{1}%
    \setlength\tabcolsep{0pt}%
    \put(0,0){\includegraphics[width=\unitlength,page=1]{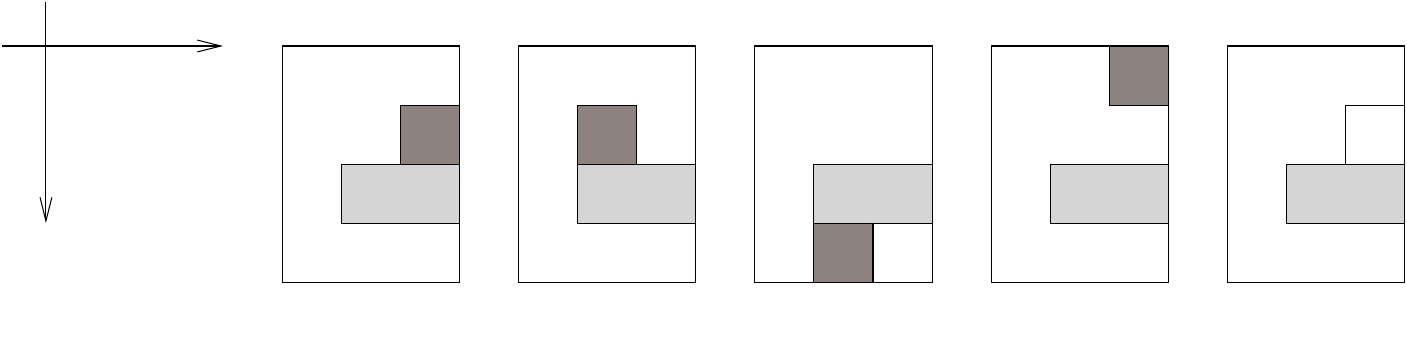}}%
    \put(0.12089246,0.18133108){\color[rgb]{0,0,0}\makebox(0,0)[lt]{\lineheight{1.25}\smash{\begin{tabular}[t]{l}$y$\end{tabular}}}}%
    \put(0.05787901,0.06790688){\color[rgb]{0,0,0}\makebox(0,0)[lt]{\lineheight{1.25}\smash{\begin{tabular}[t]{l}$x$\end{tabular}}}}%
    \put(0.26372293,0.00909433){\color[rgb]{0,0,0}\makebox(0,0)[lt]{\lineheight{1.25}\smash{\begin{tabular}[t]{l}$+1$\end{tabular}}}}%
    \put(0.59979462,0.00909433){\color[rgb]{0,0,0}\makebox(0,0)[lt]{\lineheight{1.25}\smash{\begin{tabular}[t]{l}$+1$\end{tabular}}}}%
    \put(0.76783047,0.00909433){\color[rgb]{0,0,0}\makebox(0,0)[lt]{\lineheight{1.25}\smash{\begin{tabular}[t]{l}$-1$\end{tabular}}}}%
    \put(0.93586632,0.00909433){\color[rgb]{0,0,0}\makebox(0,0)[lt]{\lineheight{1.25}\smash{\begin{tabular}[t]{l}$-1$\end{tabular}}}}%
    \put(0.43175878,0.00909433){\color[rgb]{0,0,0}\makebox(0,0)[lt]{\lineheight{1.25}\smash{\begin{tabular}[t]{l}$-1$\end{tabular}}}}%
  \end{picture}%
\endgroup%

    \caption{Examples of effects for pairs of dominoes}
    \label{fig:calctw}
\end{figure}

We then define the twist $\Tw(\bt)$ as
$\frac{1}{4}$ times the sum of all the effects
of $d_1$ on $d_0$ for all the suitable pairs $(d_0,d_1)$.
If $\cR$ is a cylinder and $\bt$ is a tiling of $\cR$
then $\Tw(\bt)$ is an integer \cite{segundoartigo}.

For example, the twist of the tiling in Figure~\ref{fig:cobtw} is $-1$,
since each floor has a contribution of $-\frac{1}{2}$.
If we appy a trit to this tiling,
the twist of the new tiling is $0$.
In Figure~\ref{fig:415} we show two examples of tilings of cylinders,
and in Figure~\ref{fig:413} we rotate each of them,
noting the twist is unchanged,
consistently with the remarks in the Introduction.

\begin{figure}[ht]
\centering
 \includegraphics[scale=0.3]{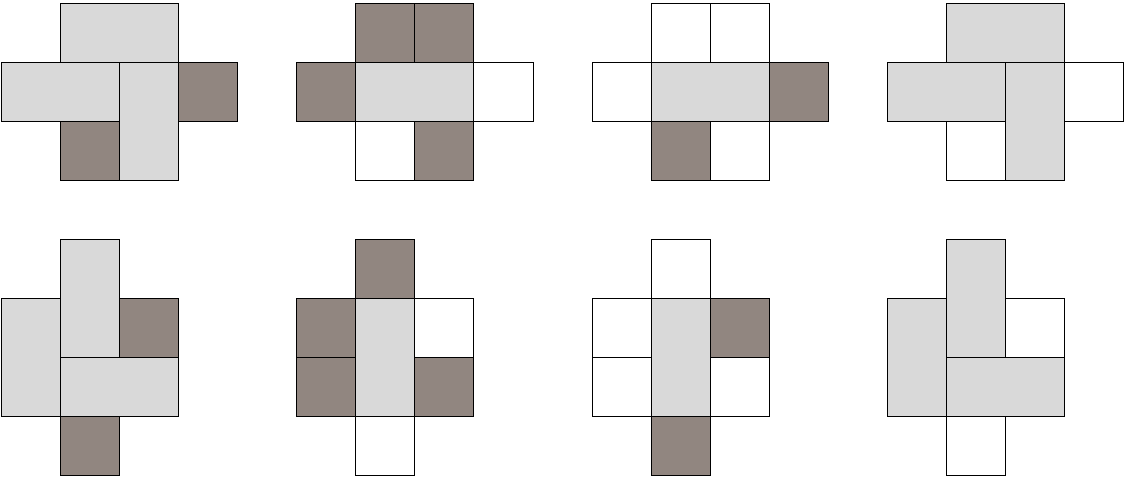}
\caption{Two examples of tilings of cylinders;
in both cases, $\Tw = 0$}
\label{fig:415}
\end{figure}

\begin{figure}[ht]
\centering
  \includegraphics[scale=0.3]{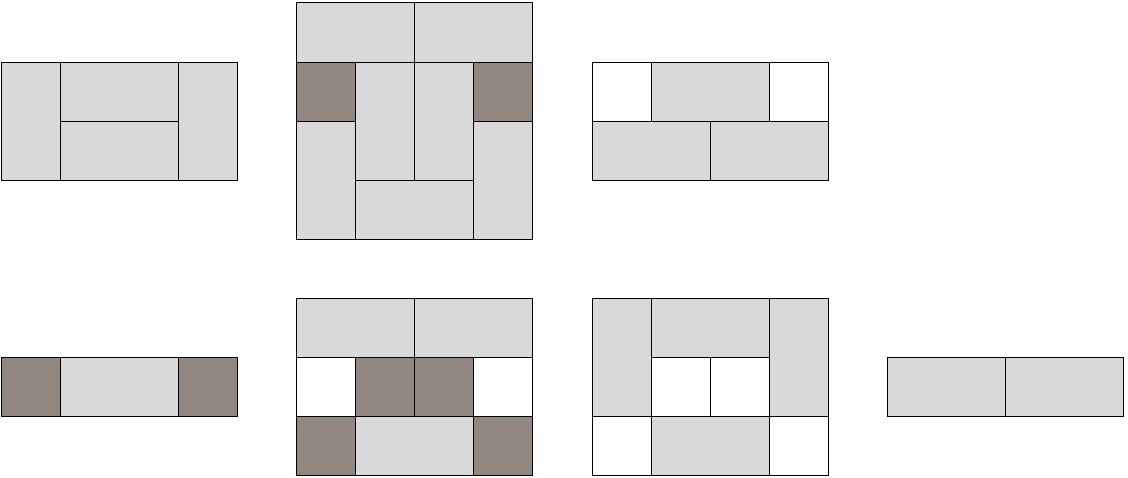}

\caption{The $x$-view of the two tilings in Figure~\ref{fig:415},
i.e., the same tilings after a rotation;
we still have $\Tw=0$}
\label{fig:413}
\end{figure}


\section{Mixed tilings}
\label{section:mixed}

The study of slab tilings led us to consider 
another class of tilings: \textit{mixed tilings}.
Given a three dimensional region $\cR \subset \RR^3$
with a specified vertical direction,
a mixed tiling is a decomposition of $\cR$
into horizontal slabs and vertical dominoes.
We admit as mixed tilings
the configurations where there are
only horizontal slabs or only vertical dominoes. 

\begin{figure}[ht]
    \centering
\includegraphics[height=1.5cm]{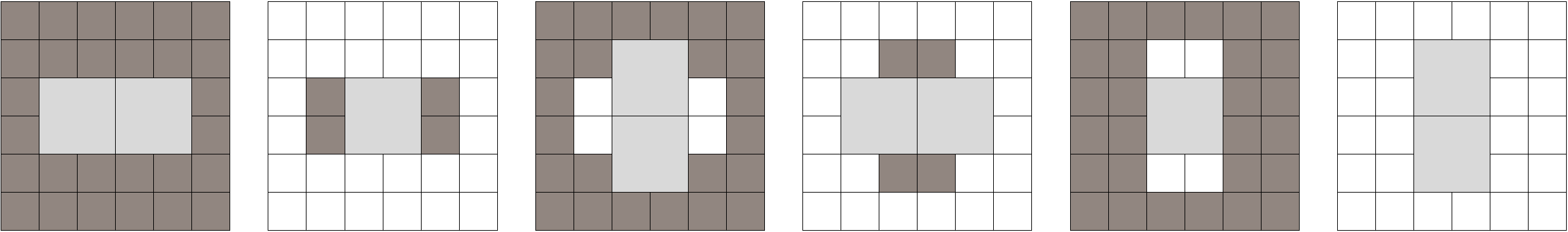}

    \caption{A non trivial example of mixed tiling}
    \label{fig:mixed}
\end{figure}

The mixed tiling in Figure~\ref{fig:mixed} can be obtained
from the slab tiling in Figure~\ref{fig:cob}
by slicing each vertical slab in half along a plane perpendicular to the paper.
This process of slicing vertical slabs in half
to transform a slab tiling into a mixed tiling can always be done.
However, it is not always possible to undo this process
in a well defined manner.
For example, in Figure~\ref{fig:mixed}
there are many ways to join the vertical dominoes
which cross the first and second floor.
Also, there are mixed tilings which do not come from any slab tiling,
as in Figure~\ref{fig:mixednotslab}.

\begin{figure}[ht]
\centering
\includegraphics[height=1.5cm]{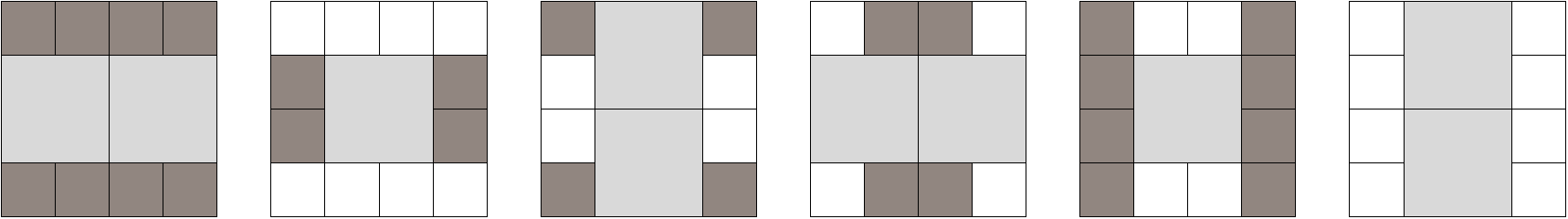}

\caption{A mixed tiling which does not correspond to a slab tiling,
obtained from the tiling in Figure~\ref{fig:mixed}
by taking the central $4\times 4\times 6$ box}
\label{fig:mixednotslab}
\end{figure}

We can also define a flip in a mixed tiling.
In one way, we take two adjacent slabs which lie on different floors,
perform a slab flip and then slice the resulting vertical slabs in half.
In the other way, we take four vertical dominoes
which form a $2\times 2\times 2$ cube,
join two pairs of adjacent dominoes
and perform a slab flip to end with two horizontal slabs.
The mixed tiling in Figure~\ref{fig:mixednotslab} admits no flip;
See Figure~\ref{fig:flipmixed} for an example of a flip.

\begin{figure}[ht]
    \centering
\includegraphics[height=1.5cm]{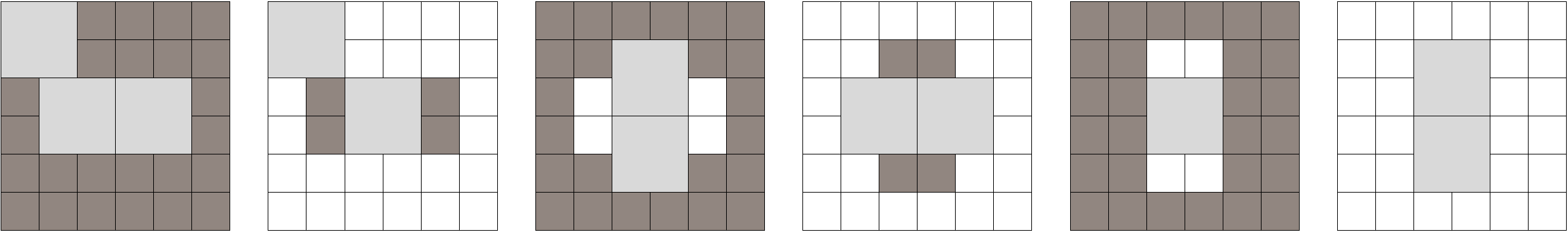}

    \caption{We show one flip performed on the tiling of Figure~\ref{fig:mixed}}
    \label{fig:flipmixed}
\end{figure}


\section{Coloring and twist}
\label{section:coloring}

The goal of this section is to describe
the process of transforming a slab or mixed tiling $\bt$
of a given cylinder $\cR$ into
a domino tiling $\tilde{\bt}$ of another cylinder $\tilde{\cR}$.
This will be used in the end to define the twist of a mixed tiling
and the triple twist of a slab tiling.

The first step of this process is to consider a certain coloring
of the cubes which compose our cylinder.
The coloring uses four colors and must meet the condition
that each slab covers exactly one cube of each color.
Provisionally, this condition will be called \textit{slab-type}.
We use the colors green, blue, red and yellow.

\begin{figure}[ht]
    \centering
\def\svgwidth{12cm}
\begingroup%
  \makeatletter%
  \providecommand\color[2][]{%
    \errmessage{(Inkscape) Color is used for the text in Inkscape, but the package 'color.sty' is not loaded}%
    \renewcommand\color[2][]{}%
  }%
  \providecommand\transparent[1]{%
    \errmessage{(Inkscape) Transparency is used (non-zero) for the text in Inkscape, but the package 'transparent.sty' is not loaded}%
    \renewcommand\transparent[1]{}%
  }%
  \providecommand\rotatebox[2]{#2}%
  \newcommand*\fsize{\dimexpr\f@size pt\relax}%
  \newcommand*\lineheight[1]{\fontsize{\fsize}{#1\fsize}\selectfont}%
  \ifx\svgwidth\undefined%
    \setlength{\unitlength}{1334bp}%
    \ifx\svgscale\undefined%
      \relax%
    \else%
      \setlength{\unitlength}{\unitlength * \real{\svgscale}}%
    \fi%
  \else%
    \setlength{\unitlength}{\svgwidth}%
  \fi%
  \global\let\svgwidth\undefined%
  \global\let\svgscale\undefined%
  \makeatother%
  \begin{picture}(1,0.14992504)%
    \lineheight{1}%
    \setlength\tabcolsep{0pt}%
    \put(0,0){\includegraphics[width=\unitlength,page=1]{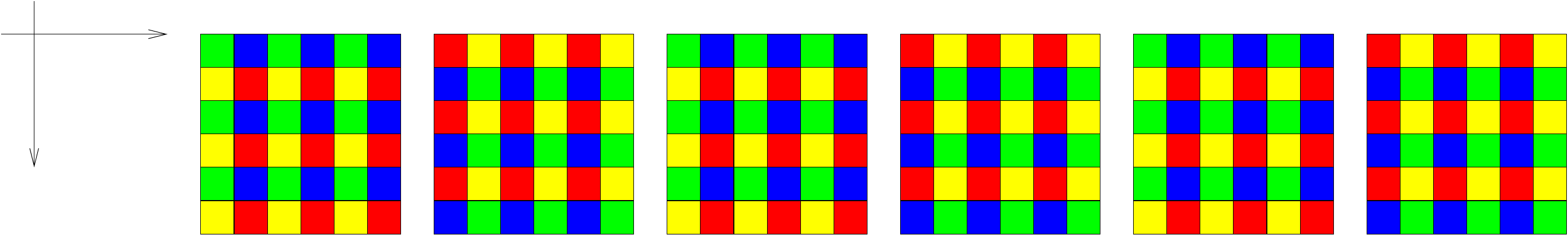}}%
    \put(0.03882119,0.02608182){\color[rgb]{0,0,0}\makebox(0,0)[lt]{\lineheight{1.25}\smash{\begin{tabular}[t]{l}$x$\end{tabular}}}}%
    \put(0.08132615,0.10259074){\color[rgb]{0,0,0}\makebox(0,0)[lt]{\lineheight{1.25}\smash{\begin{tabular}[t]{l}$y$\end{tabular}}}}%
  \end{picture}%
\endgroup%

    \caption{A slab-type coloring of a $6\times 6\times 6$ box,
    based on Example~\ref{example:slabtype}}
    \label{fig:slab}
\end{figure}

\begin{example}
\label{example:slabtype}
Here we give an algebraic description of a slab-type coloring.
Let $\phi: \ZZ \oplus \ZZ \oplus \ZZ \to \ZZ/(2) \oplus \ZZ/(2)$
be the additive group homomorphism
$\phi(x,y,z) = (y+z,x+z)$.
Associate with each element of $\ZZ/(2) \oplus \ZZ/(2)$ a color
and paint the cube $[x,x+1]\times [y,y+1] \times [z,z+1]$
with the color $\phi(x,y,z)$.
As in Figure~\ref{fig:slab},
$(0,0)$ is green, $(1,0)$ is yellow,
$(0,1)$ is blue and $(1,1)$ is red.
\end{example}


\smallskip
\begin{lemma}
\label{lemma:unique}
Let $\cR=[0,L]\times[0,M]\times[0,N]$ be a box with $\min\{L,M,N\}\geq 2$.
There is, up to permutation of the colors,
only one slab type coloring.
\end{lemma}

\begin{proof}
Refer to Figure~\ref{fig:color}.
Consider a slab which covers the yellow, red, $Z$ and $W$ squares. This means $W$ must be painted either blue or green. Now consider a slab which covers the blue, red, $Y$ and $W$ squares. The only possibility left is for $W$ to be painted green. Analogously, $X$ must be red, $Y$ must be yellow and $Z$ must be blue.
    
    In a similar fashion, we rule out possibilities to propagate the painting to the remaining part of $\cR$.
\end{proof}

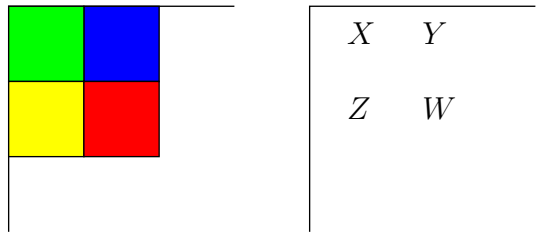
\begin{figure}[ht]
        \centering
\begingroup%
  \makeatletter%
  \providecommand\color[2][]{%
    \errmessage{(Inkscape) Color is used for the text in Inkscape, but the package 'color.sty' is not loaded}%
    \renewcommand\color[2][]{}%
  }%
  \providecommand\transparent[1]{%
    \errmessage{(Inkscape) Transparency is used (non-zero) for the text in Inkscape, but the package 'transparent.sty' is not loaded}%
    \renewcommand\transparent[1]{}%
  }%
  \providecommand\rotatebox[2]{#2}%
  \newcommand*\fsize{\dimexpr\f@size pt\relax}%
  \newcommand*\lineheight[1]{\fontsize{\fsize}{#1\fsize}\selectfont}%
  \ifx\svgwidth\undefined%
    \setlength{\unitlength}{200bp}%
    \ifx\svgscale\undefined%
      \relax%
    \else%
      \setlength{\unitlength}{\unitlength * \real{\svgscale}}%
    \fi%
  \else%
    \setlength{\unitlength}{\svgwidth}%
  \fi%
  \global\let\svgwidth\undefined%
  \global\let\svgscale\undefined%
  \makeatother%
  \begin{picture}(1,0.435)%
    \lineheight{1}%
    \setlength\tabcolsep{0pt}%
    \put(0,0){\includegraphics[width=\unitlength,page=1]{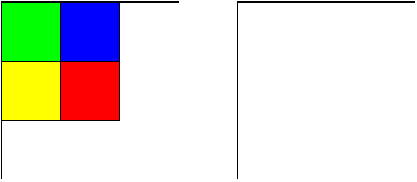}}%
    \put(0.64177691,0.35927693){\color[rgb]{0,0,0}\makebox(0,0)[lt]{\lineheight{1.25}\smash{\begin{tabular}[t]{l}$X$\end{tabular}}}}%
    \put(0.78355385,0.35927693){\color[rgb]{0,0,0}\makebox(0,0)[lt]{\lineheight{1.25}\smash{\begin{tabular}[t]{l}$Y$\end{tabular}}}}%
    \put(0.64177691,0.2175){\color[rgb]{0,0,0}\makebox(0,0)[lt]{\lineheight{1.25}\smash{\begin{tabular}[t]{l}$Z$\end{tabular}}}}%
    \put(0.78355385,0.2175){\color[rgb]{0,0,0}\makebox(0,0)[lt]{\lineheight{1.25}\smash{\begin{tabular}[t]{l}$W$\end{tabular}}}}%
  \end{picture}%
\endgroup%

        \caption{The top left corner of the first two floors of $\cR$}
        \label{fig:color}
\end{figure}

\begin{remark}
\label{remark:unique}
The uniqueness property in Lemma \ref{lemma:unique}
also holds for other kinds of regions and does not hold for some others.
We do not think a more precise statement would be useful here.

Given Lemma \ref{lemma:unique},
there is limited interest in considering more general colorings
if, as in the present paper, our focus is on boxes.
From now on, we will 
consider only the coloring in Example \ref{example:slabtype}.
\end{remark}

The second step of the process is to transform the region.
We first consider cylinders of the form
$\cR = \cD \times [0,N]$ and the $z$-view.
A pair of colors $(c_0,c_1)$ is \textit{good} (for this view)
if $c_0$ and $c_1$ are on a same diagonal on the first floor:
the good pairs are therefore red-green and blue-yellow.
The colors in a good pair are also called \textit{good};
the other colors are \textit{bad}.
Informally,
we keep all cubes (or squares) of one of the two good colors
and delete all cubes of bad colors.
We then inflate the remaining cubes so as to obtain cuboids
which fill in the space left by the deleted cubes.
Figure~\ref{fig:slabd} shows an example:
the initial region is as in Figure~\ref{fig:slab}
and the good pair is red-green.

\begin{figure}[ht]
    \centering
  \includegraphics[scale=0.25]{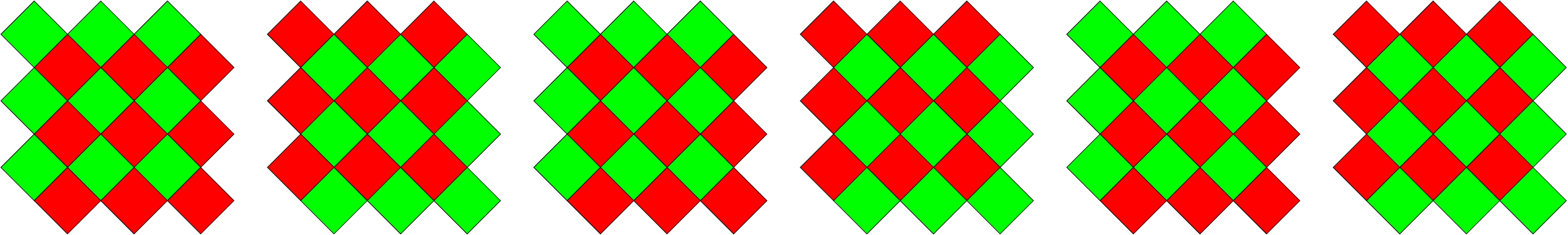}

    \caption{The region in Figure~\ref{fig:slab} is transformed
    by deleting the blue and yellow cubes and
    inflating the red and green ones}
    \label{fig:slabd}
\end{figure}

A more precise construction is in order.
For the red-green pair, the cube 
$[x,x+1]\times[y,y+1]\times[z,z+1]$ is good if $x+y$ is even.
In the notation of Example~\ref{example:slabtype},
the cube is good if
$\phi(x,y,z) \in \{(0,0),(1,1)\} \subset \ZZ/(2) \oplus \ZZ/(2)$.
If good, the above cube inflates to become the cuboid
$\tilde S \times [z,z+1]$ where $\tilde S \subset \RR^2$
is the square with vertices
$(x-\frac12,y+\frac12)$,
$(x+\frac12,y-\frac12)$,
$(x+\frac32,y+\frac12)$,
$(x+\frac12,y+\frac32)$.
The region $\tilde R$ is the union of the resulting inflated cuboids,
as in Figure~\ref{fig:slabd}.
The new cuboids keep the old colors,
so that $\tilde R$ is a cubiculated region.
Strictly speaking, we might feel the need to rotate and rescale,
but we usually refrain from doing so as it only complicates the description.


We are now left with the task of transforming the tilings.
We first consider slab tilings; let $\bt$ be a tiling of $\cR$.
Each slab $P$ of $\bt$ covers two preserved cubes and two deleted cubes.
Let $S_{0}$ and $S_{1}$ be the two preserved cubes,
which are transformed into $\tilde S_{0}$ and $\tilde S_{1}$, respectively.
We say that $P$ is taken to the domino $\tilde P$,
the union of $\tilde S_{0}$ and $\tilde S_{1}$.
See Figures~\ref{fig:cob} and \ref{fig:cobd} for an example.
In the second row of Figure~\ref{fig:cobd},
we show the tiling rotated clockwise by $45^{\circ}$;
in this example, we have $\Tw = +2$.

\begin{figure}[ht]
    \centering
 \includegraphics[scale=0.25]{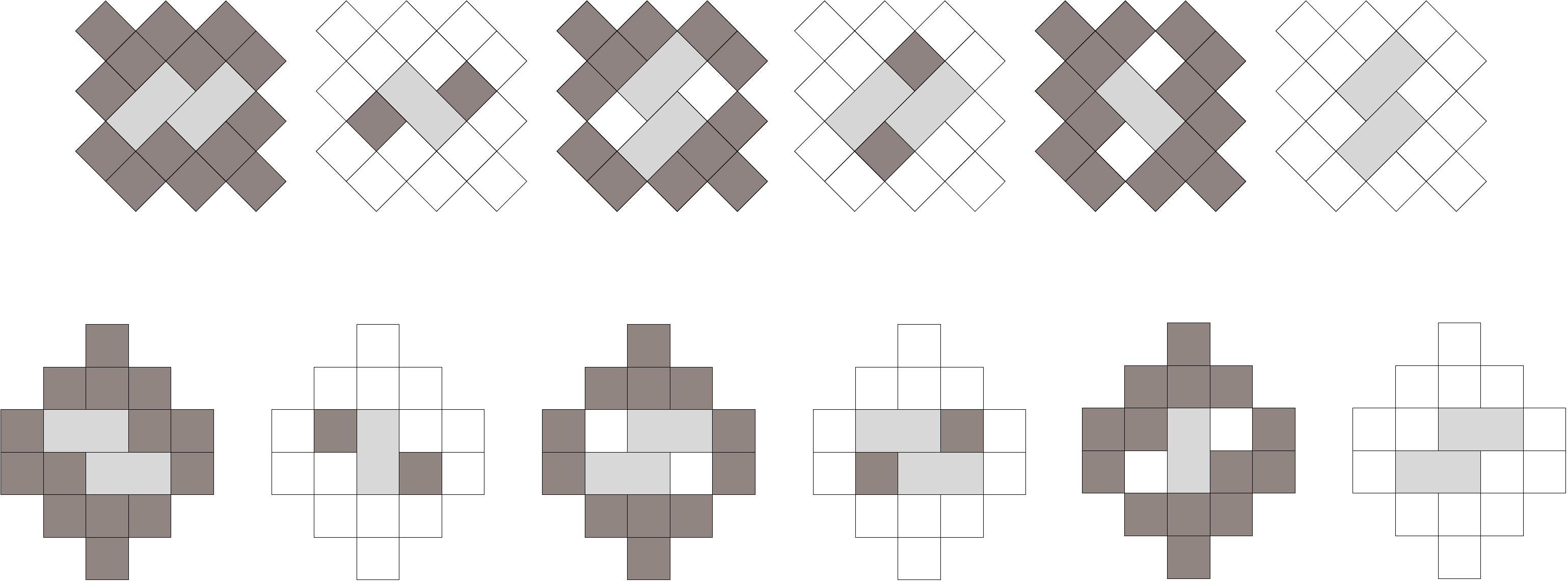}

     \caption{The domino tiling obtained
from the slab tiling in Figure~\ref{fig:cob} or
the mixed tiling in Figure~\ref{fig:mixed}}
    \label{fig:cobd}
\end{figure}


For mixed tilings the construction is similar.
Horizontal slabs are transformed into horizontal dominoes, as above.
Vertical dominoes consist of two cubes which are
either both preserved or both removed.
In the first case, the domino in transformed into a vertical domino;
in the second case the domino is deleted;
see Figures~\ref{fig:mixed} and \ref{fig:cobd}.

Now we consider $\bt$ a slab tiling of a box.
Until this moment, we have only looked at our tiling $\bt$ of $\cR$
from a particular position of the axes,
as stated at the beginning of the section.
This position of the axes is exactly that of an observer
looking from the positive part of the $z$-axis.
We say we are looking at the \textit{$z$-view} $\bt_{z}$ of $\bt$.
What we can actually do is look at the region $\cR$
from the positions of the positive parts of the $x$ and $y$ axes,
yielding the views $\bt_{x}$ and $\bt_{y}$, respectively.
Strictly speaking, $\bt_{x}$, $\bt_{y}$ and $\bt_{z}$ are
three different views of the same tiling (as sets of slabs).
What does change though is the floor diagram.

\begin{figure}[ht]
    \centering
  \def\svgwidth{12cm}
\begingroup%
  \makeatletter%
  \providecommand\color[2][]{%
    \errmessage{(Inkscape) Color is used for the text in Inkscape, but the package 'color.sty' is not loaded}%
    \renewcommand\color[2][]{}%
  }%
  \providecommand\transparent[1]{%
    \errmessage{(Inkscape) Transparency is used (non-zero) for the text in Inkscape, but the package 'transparent.sty' is not loaded}%
    \renewcommand\transparent[1]{}%
  }%
  \providecommand\rotatebox[2]{#2}%
  \newcommand*\fsize{\dimexpr\f@size pt\relax}%
  \newcommand*\lineheight[1]{\fontsize{\fsize}{#1\fsize}\selectfont}%
  \ifx\svgwidth\undefined%
    \setlength{\unitlength}{1391bp}%
    \ifx\svgscale\undefined%
      \relax%
    \else%
      \setlength{\unitlength}{\unitlength * \real{\svgscale}}%
    \fi%
  \else%
    \setlength{\unitlength}{\svgwidth}%
  \fi%
  \global\let\svgwidth\undefined%
  \global\let\svgscale\undefined%
  \makeatother%
  \begin{picture}(1,0.57225018)%
    \lineheight{1}%
    \setlength\tabcolsep{0pt}%
    \put(0,0){\includegraphics[width=\unitlength,page=1]{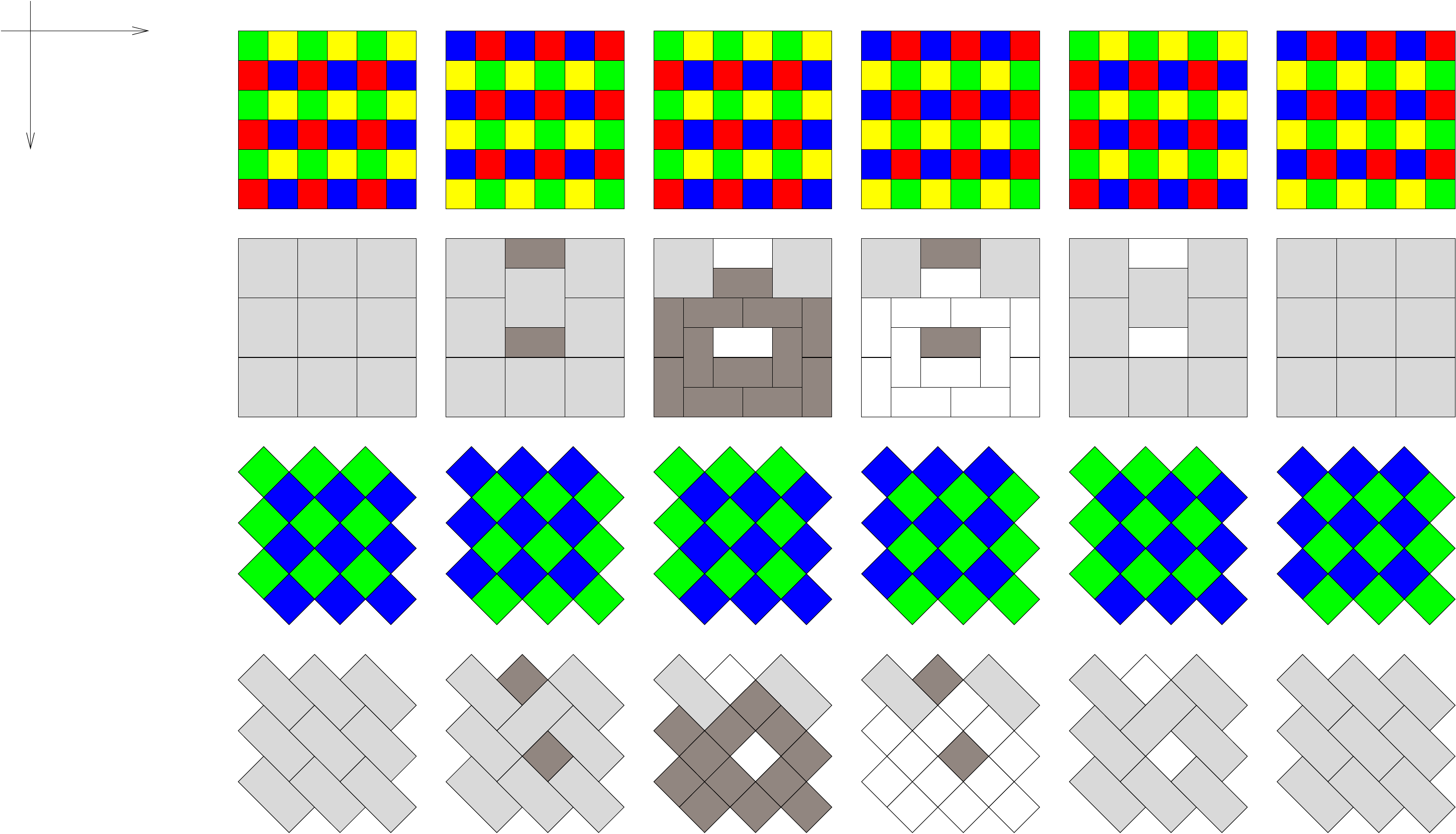}}%
    \put(0.03723838,0.45329007){\color[rgb]{0,0,0}\makebox(0,0)[lt]{\lineheight{1.25}\smash{\begin{tabular}[t]{l}$z$\end{tabular}}}}%
    \put(0.07801033,0.52667957){\color[rgb]{0,0,0}\makebox(0,0)[lt]{\lineheight{1.25}\smash{\begin{tabular}[t]{l}$x$\end{tabular}}}}%
  \end{picture}%
\endgroup%

    \caption{The figures corresponding to the $y$-view,
with the blue-green good pair; $\Tw(\tilde\bt) = 0$}
    \label{fig:verdeazul}
\end{figure}

\begin{figure}[ht]
    \centering
 \def\svgwidth{12cm}
\begingroup%
  \makeatletter%
  \providecommand\color[2][]{%
    \errmessage{(Inkscape) Color is used for the text in Inkscape, but the package 'color.sty' is not loaded}%
    \renewcommand\color[2][]{}%
  }%
  \providecommand\transparent[1]{%
    \errmessage{(Inkscape) Transparency is used (non-zero) for the text in Inkscape, but the package 'transparent.sty' is not loaded}%
    \renewcommand\transparent[1]{}%
  }%
  \providecommand\rotatebox[2]{#2}%
  \newcommand*\fsize{\dimexpr\f@size pt\relax}%
  \newcommand*\lineheight[1]{\fontsize{\fsize}{#1\fsize}\selectfont}%
  \ifx\svgwidth\undefined%
    \setlength{\unitlength}{1391bp}%
    \ifx\svgscale\undefined%
      \relax%
    \else%
      \setlength{\unitlength}{\unitlength * \real{\svgscale}}%
    \fi%
  \else%
    \setlength{\unitlength}{\svgwidth}%
  \fi%
  \global\let\svgwidth\undefined%
  \global\let\svgscale\undefined%
  \makeatother%
  \begin{picture}(1,0.57225018)%
    \lineheight{1}%
    \setlength\tabcolsep{0pt}%
    \put(0,0){\includegraphics[width=\unitlength,page=1]{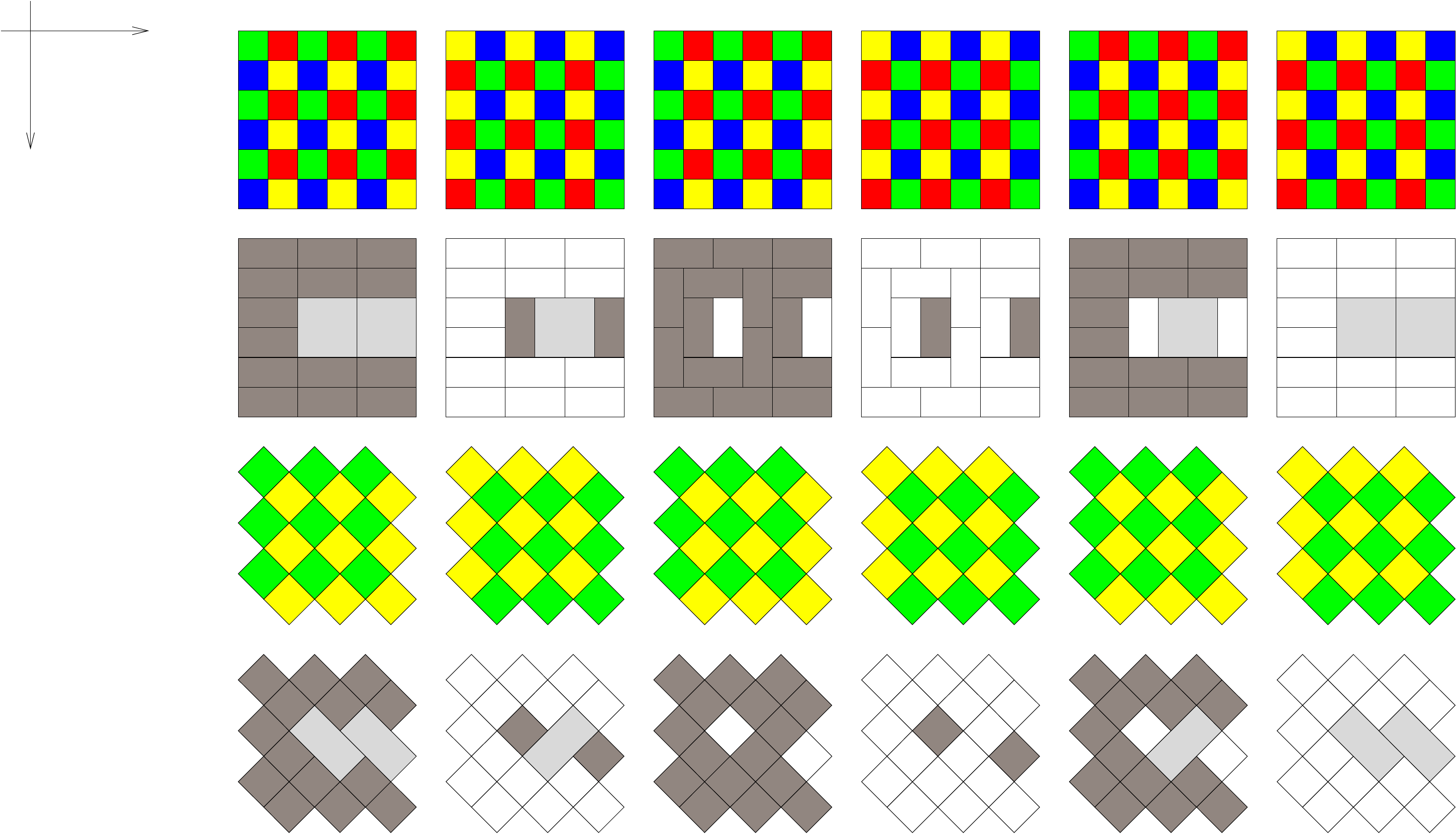}}%
    \put(0.03723838,0.45329007){\color[rgb]{0,0,0}\makebox(0,0)[lt]{\lineheight{1.25}\smash{\begin{tabular}[t]{l}$y$\end{tabular}}}}%
    \put(0.07801033,0.52667957){\color[rgb]{0,0,0}\makebox(0,0)[lt]{\lineheight{1.25}\smash{\begin{tabular}[t]{l}$z$\end{tabular}}}}%
  \end{picture}%
\endgroup%

    \caption{The figures corresponding to the $x$-view,
with the green-yellow good pair; $\Tw(\tilde\bt) = 0$}
    \label{fig:verdeamarelo}
\end{figure}

For each view there are two good pairs of colors.
For the $y$-view, the good pairs are green-blue and red-yellow;
for the $x$-view, red-blue and green-yellow.
In general, a pair of colors $\kappa$ defines an axis
and an axis defines two complementary pairs of colors, 
$\kappa$ and $\bar\kappa$.

Given a view, a good pair of colors and a tiling of a box,
we perform the above construction.
In Figures \ref{fig:verdeazul} and \ref{fig:verdeamarelo}
we continue the example of the slab tiling
of the $6\times6\times6$ box given in Figure~\ref{fig:cob}.
Notice that the resulting regions and
domino tilings $\tilde\bt$ are quite different.
When necessary, we write $\tilde\cR_{\kappa}$ and $\tilde\bt_{\kappa}$
to indicate the pair of colors.


We are ready to state the definitions of the twist.

\begin{figure}[t]
    \centering
    \includegraphics[scale=0.2]{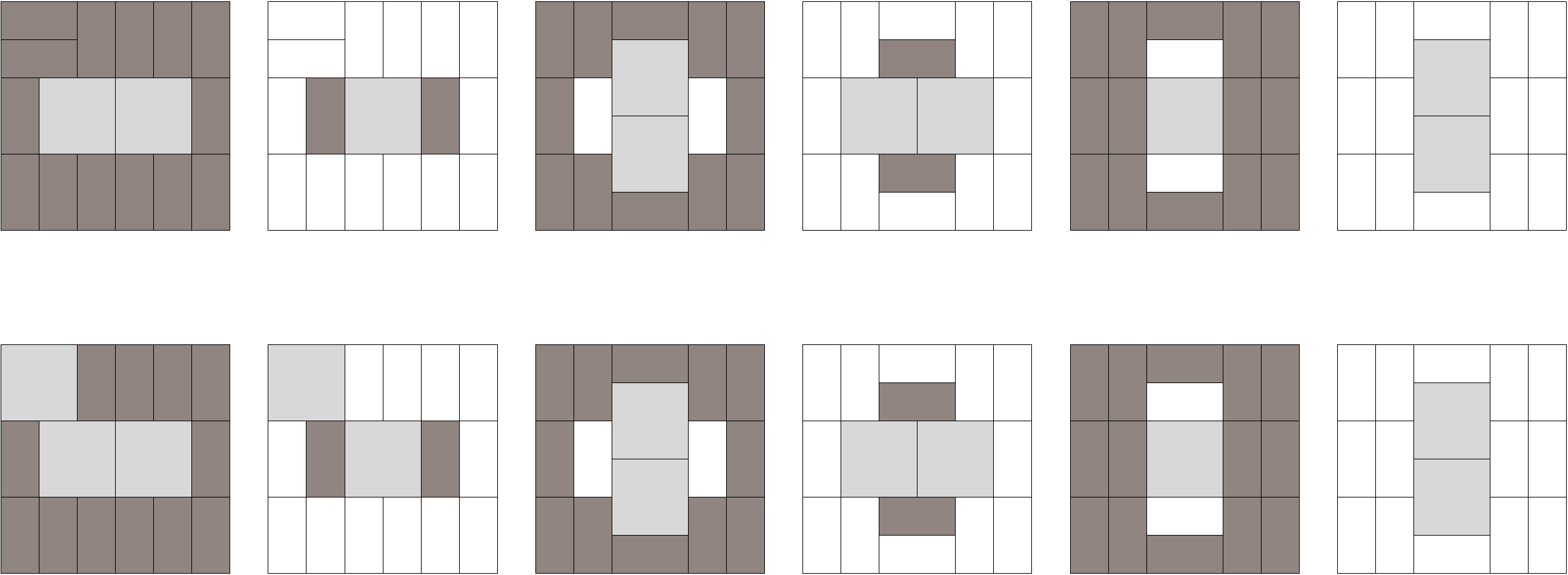}
    \caption{Two possibilities of flipping two slabs
    on the top left corner of the first floor of Figure~\ref{fig:cob};
    under the transformation,
    the first possibility yields
    the same domino tiling as in Figure~\ref{fig:cobd}}
    \label{fig:cob23}
\end{figure}

\begin{figure}[t]
    \centering
   \includegraphics[scale=0.2]{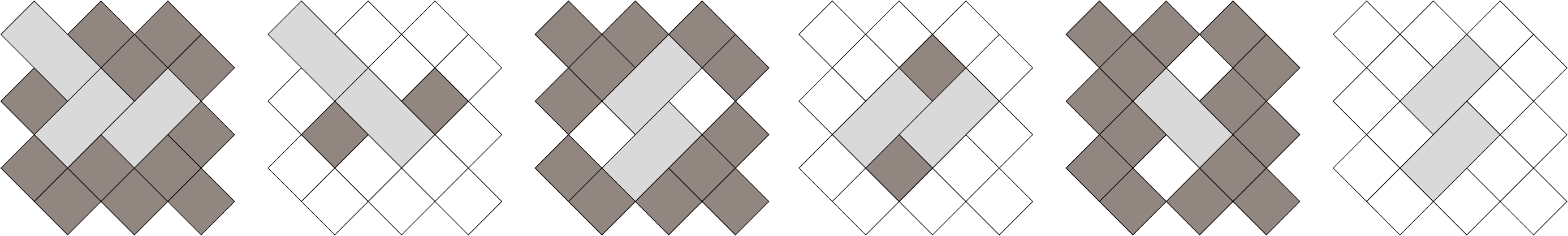}

    \caption{The domino tiling from the second possibility in Figure~\ref{fig:cob23} after the transformation}
    \label{fig:cobd2}
\end{figure}

\medskip  

\begin{defin}[Mixed twist]
Let $\cR$ be a $z$-cylinder and $\kappa$ be a good pair of colors.
Let $\bt$ be a mixed tiling of $\cR$.
Consider the transformed region $\tilde\cR_\kappa$
and the corresponding domino tiling $\tilde{\bt}_\kappa$.
The \textit{twist} of $\bt$ is
\[\Tw_\kappa(\bt)=\Tw(\tilde{\bt}_\kappa)\in\mathbb{Z}.\]
\end{defin}

\medskip  

\begin{defin}[Triple twist]
Let $\cR$ be a box and let $\bt$ be a slab tiling of $\cR$.
Let $\kappa_x, \kappa_y, \kappa_z$ be good pairs of colors
for the $x$, $y$ and $z$ views, respectively.
Consider the transformed domino tilings
$\tilde \bt_{\kappa_x}$, $\tilde \bt_{\kappa_y}$ and $\tilde \bt_{\kappa_z}$.
Let $\Tw_\kappa(\bt)=\Tw(\tilde \bt_\kappa)$.
The \textit{triple twist} of $\bt$ is
\[ \TTw_{\kappa_x,\kappa_y,\kappa_z}(\bt) = 
(\Tw_{\kappa_x}(\bt),\Tw_{\kappa_y}(\bt),\Tw_{\kappa_z}(\bt))\in \mathbb{Z}^3. \]
\end{defin}

\bigskip


\section{Properties of the Triple Twist}
\label{section:properties}

The twist for domino tilings has several simple properties and 
some of them are inherited by the triple twist.

\medskip

Let $\bt$ be a slab or mixed tiling of a cylinder $\cR$.
Fix a good pair of colors $\kappa$.
When performing a flip in $\bt$,
one of two things happens with the dominoes
in the corresponding domino tiling $\tilde\bt$:
either a domino flip is performed or nothing changes at all.
Since the domino twist is flip invariant,
this means that the several $\Tw_{\kappa}$ defined above
are invariants under flips.

The complement $\bar\kappa$ of a good pair of colors $\kappa$ is also good.
One of the first questions which arises
when defining the triple twist
is regarding the choice of $\bar\kappa$ instead of $\kappa$.
For example, the calculation of the twist in Figure~\ref{fig:cobd}
depended on the choice of the green-red pair in Figure~\ref{fig:slab}.
We could have chosen instead the pair blue-yellow.
This other choice would even yield a different domino tileable region,
but does not actually provide any new information,
which is the content of Lemma \ref{independence}.

\smallskip
\begin{lemma}
\label{independence}
Let $\cR$ be a cylinder.
Let $\kappa$ and $\bar\kappa$ be complementary good pairs of colors.
Let $\bt$ be a mixed tiling of $\cR$.
Then, $\Tw_\kappa(\bt)= -\Tw_{\bar\kappa}(\bt)$.
\end{lemma}

\begin{proof}
Assume that the two good pairs are $\kappa$ for red-green
and $\bar\kappa$ for yellow-blue.
Throughout the proof, when talking about the overlap of two pieces,
we refer to the overlap of their projections onto the $xy$ plane.
Firstly, we take a look at how two horizontal slabs
in different floors overlap.
There are three possibilities:
they can overlap in either one, two or four squares.
As an illustrative example,
we show these three cases in the mixed tiling
of the $4\times 4\times 4$ box in Figure~\ref{fig:411}.

\begin{figure}[ht]
    \centering
  \includegraphics[scale=0.4]{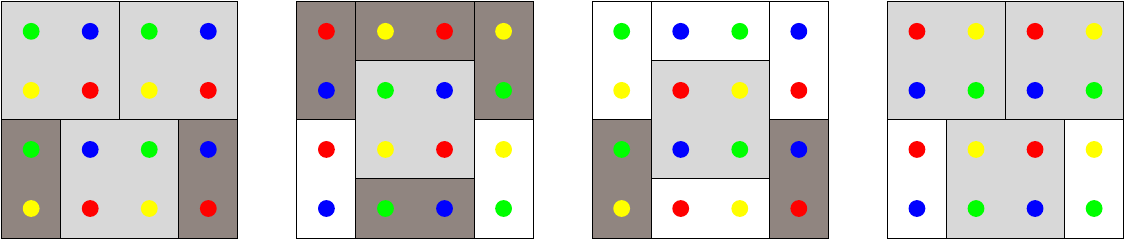}

    \caption{An example of a mixed tiling;
    the top left cube on the first floor is $[0,1]^3$,
    which is green}
    \label{fig:411}
\end{figure}

\begin{figure}[ht]
    \centering
  \includegraphics[scale=0.3]{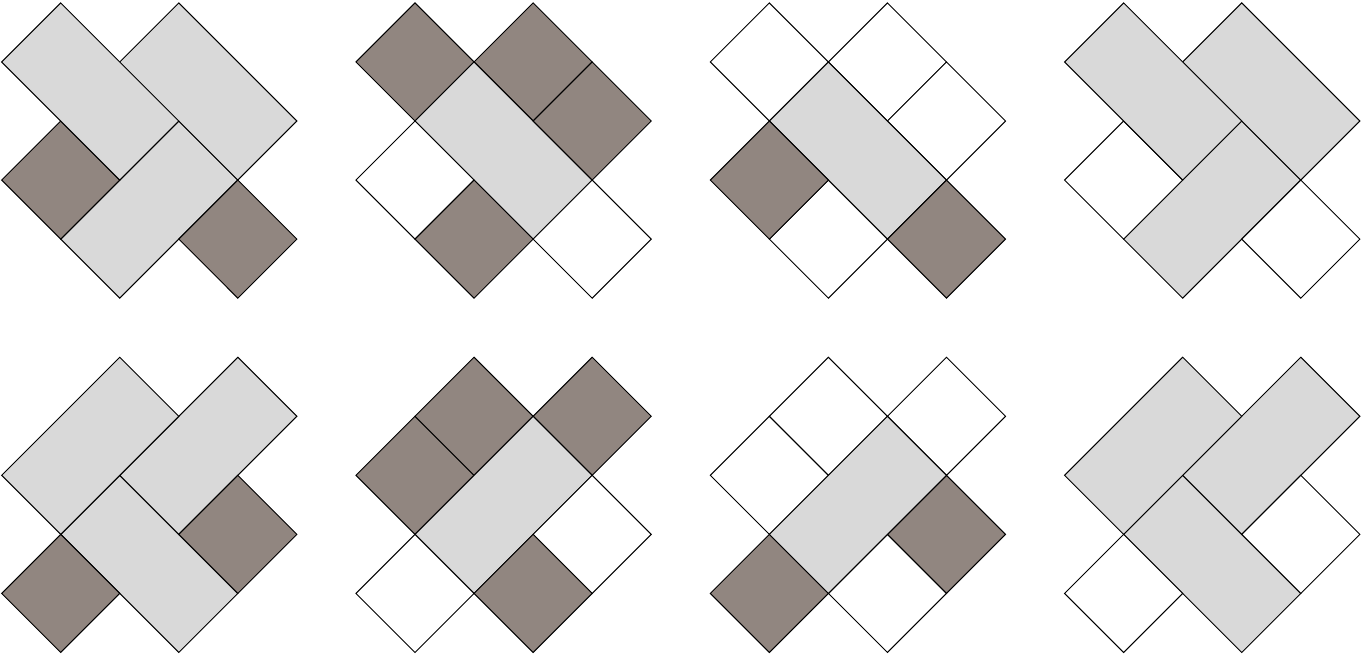}

    \caption{First tiling is red-green and second tiling is blue-yellow}
    \label{fig:412}
\end{figure}
    
We transform the mixed tiling $\bt$ into two domino tilings
$\tilde\bt_{\kappa}$ and $\tilde\bt_{\bar\kappa}$.
In Figure~\ref{fig:412} we show these two tilings
for the example in Figure~\ref{fig:411}.
We will calculate the two twists side by side, but in order to do so,
we will rotate both regions by $45$ degrees as in Figure~\ref{fig:415}
and then look at them from the $x$-axis view,
as depicted in Figure~\ref{fig:413}.
By doing so, dominoes which were vertical in Figure~\ref{fig:411}
(and thus vertical in Figure~\ref{fig:412})
become horizontal in Figure~\ref{fig:415}
and hence will not affect the calculation of the twists.
    
In Figure~\ref{fig:414} we look at the relative position
of a pair of dominoes which comes from a pair of overlapping slabs
(on different floors) in $\bt$
after performing the sequence of transformations in the above paragraph.
We do not require the pair of slabs to be on consecutive floors.
    \begin{figure}[ht]
        \centering
  \includegraphics[scale=0.3]{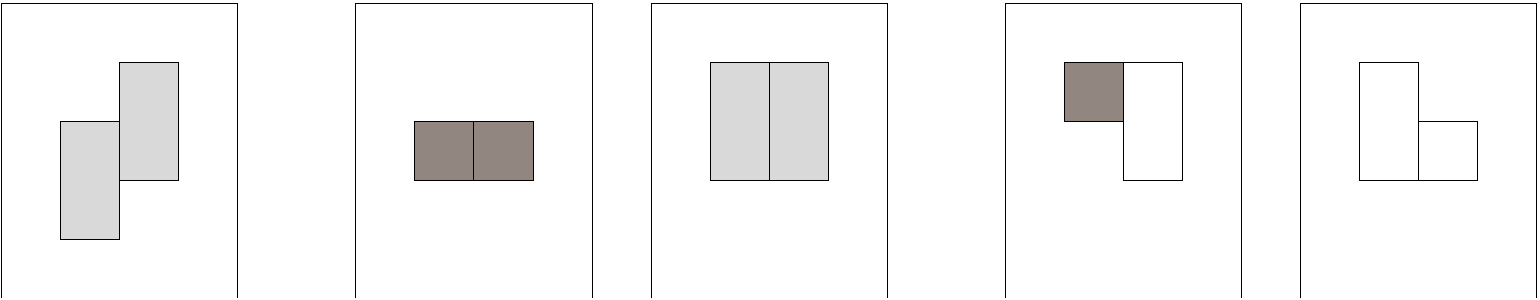}

        \caption{We have, in order:
one instance of case 1, two instances of case 2 and two instances of case 3}
        \label{fig:414}
    \end{figure}

    \begin{enumerate}
        \item If a pair of slabs overlaps in one square,
the corresponding dominoes are parallel in both
the green-red and blue-yellow tilings.
In one of the tilings the dominoes will both be vertical
with no overlap and in the other tiling both
will be horizontal with an overlap of one square.
        \item If a pair of slabs overlaps in four squares,
the corresponding dominoes will be parallel in both
the green-red and blue-yellow tilings,
but in the two cases they will be horizontal
and have an overlap of two squares.
        \item If a pair of slabs overlaps in two squares,
one of the corresponding dominoes will be vertical and the other horizontal,
overlapping on one square.
But in the green-red tiling the contribution to the twist
will be $+1$ and on the blue-yellow tiling the contribution will be $-1$
(or vice-versa).
    \end{enumerate}

Therefore, summing the contributions of the twists,
which are only nontrivial in case 3,
we see that, indeed, the statement holds.
\end{proof}

We now consider the effect of spatial symmetries
on the twist of mixed tilings.
The reader should compare the next result with the corresponding
simpler situation for twists of domino tilings.

\smallskip
\begin{prop}
\label{prop:mixedsymmetry}
Consider the box
$\cR = [-L,L] \times [-M,M] \times [-N,N]$,
where $L, M, N$ are positive integers.
Let $\bt$ be a mixed tiling of $\cR$,
with vertical dominoes in the $z$ direction.
Let $\kappa$ be a good pair of colors for the $z$-view.
\begin{enumerate}
\item{Let $r_x$, $r_y$ and $r_z$ denote the reflections
about the planes $x=0$, $y=0$ and $z=0$, respectively.
We have
\[
\Tw_\kappa(r_x(\bt)) = \Tw_\kappa(\bt), \quad
\Tw_\kappa(r_y(\bt)) = \Tw_\kappa(\bt), \quad
\Tw_\kappa(r_z(\bt)) = -\Tw_\kappa(\bt). \]}
\item{Assume that $L = M$ and
let $\rho$ denote a $90^{\circ}$ rotation around the $z$ axis.
We have $\Tw(\rho(\bt)) = -\Tw(\bt)$.}
\end{enumerate}
\end{prop}

\begin{proof}
We first consider case (item) 1 for $r_x$.
Let $\kappa$ and $\bar\kappa$ be the red-green and blue-yellow pairs,
respectively.
Let $\bt_0$ and $\bt_1 = r_x(\bt_0)$ be mixed tilings.
Let $\tilde\bt_{i,\kappa}$ and $\tilde\bt_{i,\bar\kappa}$
be the transformed tilings obtained from $\bt_i$ under each good pair.
The set of red-green cubes is taken by $r_x$
to the set of blue-yellow cubes.
Thus, the reflection $r_x$ takes $\tilde\bt_{i,\kappa}$ to
$\tilde\bt_{1-i,\bar\kappa}$ so that
$\Tw(\tilde\bt_{1-i,\bar\kappa}) = -\Tw(\tilde\bt_{i,\kappa})$.
By Lemma \ref{independence} we have
$\Tw(\tilde\bt_{i,\bar\kappa}) = -\Tw(\tilde\bt_{i,\kappa})$
and therefore
\[ \Tw_\kappa(\bt_1) =
\Tw(\tilde\bt_{1,\kappa}) = 
-\Tw(\tilde\bt_{1,\bar\kappa}) = 
\Tw(\tilde\bt_{0,\kappa}) = 
\Tw_\kappa(\bt_0), \]
as desired.
Case 1 for $r_y$ is similar.
The set of red-green cubes is taken to itself by $r_z$ 
and therefore, $\bt_3 = r_z(\bt_0)$,
case 1 for $r_z$ boils down to
\[ \Tw_\kappa(\bt_3) =
\Tw(\tilde\bt_{3,\kappa}) = 
-\Tw(\tilde\bt_{0,\kappa}) = 
-\Tw_\kappa(\bt_0), \]
completing item 1.
Finally, the rotation $\rho$
takes the set of red-green cubes to the set of blue-yellow cubes.
Set $\bt_4 = \rho(\bt_0)$:
$\tilde\bt_{4,\bar\kappa}$ is a rotation of $\tilde\bt_{0,\kappa}$
so that
$\Tw(\tilde\bt_{4,\kappa}) = -\Tw(\tilde\bt_{4,\bar\kappa}) =
-\Tw(\tilde\bt_{0,\kappa})$,
completing the proof.
\end{proof}

The following proposition is similar,
but for slab tilings.
A symmetry of $\RR^3$ which preserves unit cubes
is a map of the form $q(w) = \mathbf{v_0} + \mathbf{Qw}$ 
where $\mathbf{v_0} \in \ZZ^3$ and $\mathbf{Q}$ is a $3 \times 3$ permutation matrix with signs.
Let $\sign(q) = \det(\mathbf{Q}) \in \{\pm 1\}$.

\smallskip
\begin{prop}
\label{prop:slabsymmetry}
Consider $q$ a symmetry of $\RR^3$, as above.
Let $\cR_0$ be a cubiculated region and $\bt_0$ be a slab tiling of $\cR_0$.
Let $\cR_1 = q[\cR_0]$, a cubiculated region,
and $\bt_1 = q[\bt_0]$, a tiling of $\cR_1$.
Let $\kappa_0$ be a pair of colors and $\kappa_1 = q[\kappa_0]$.
Let $\tilde\cR_{0,\kappa_0}$ and $\tilde\cR_{1,\kappa_1}$
be the respective transformed regions, with domino tilings
$\tilde\bt_{0,\kappa_0}$ and $\tilde\bt_{1,\kappa_1}$.
We then have
\[ \Tw_{\kappa_0}(\bt_0) = \Tw(\tilde\bt_{0,\kappa_0}) =
\sign(q) \Tw(\tilde\bt_{1,\kappa_1}) = \sign(q) \Tw_{\kappa_1}(\bt_1). \]
\end{prop}

\begin{proof}
The proof is easy and similar to that of Proposition~\ref{prop:mixedsymmetry}.
The reader may want to check consistency between the two propositions
in a few cases.
\end{proof}

\bigskip


\section{Connectivity under local moves}
\label{section:localmoves}

In this section we return to our main problem:
connectivity under local moves of spaces of tilings.
We first show that the space of slab tilings 
is connected via flips for a few special regions.
We then show that
there exists no finite fixed set of local moves
which works for slab tilings of arbitrary boxes.
Equivalently, there is no bound on the number of slabs
which have to participate in each move.

\bigskip

We begin with two-floored regions.

\smallskip
\begin{lemma}
Given a quadriculated disk $\cD \subset \mathbb{R}^2$,
let $\cR = \cD \times [0,2]$,
the three dimensional region with base $\cD$ and two floors.
The space of slab tilings of $\cR$ is 
either empty or connected under flips.
\end{lemma}

\begin{proof}
Let $\tilde{\cD} \subset \RR^2$ be a bounded region
which can be tiled by $2\times 2$ squares
(the region $\tilde\cD$ need not be connected).
We claim that
there is only one way of tiling $\tilde{\cD}$ with such squares.
The proof of the claim
is by induction on the number of unit squares of $\tilde{\cD}$.
The case where $\tilde{\cD}$ comprises four unit squares is obvious.
For the inductive step, notice there is only one way of placing a slab
on a particular corner of $\tilde{\cD}$,
and then remove the four corresponding unit squares.

We return to $\cR$. We may speak in terms of the floor diagram.
On the first floor, omit all dominoes (i.e., parts of vertical slabs).
The remaining planar region $\tilde\cD_{1}$
must be tiled by $2\times 2$ squares,
which can only be done in one way.
On the second floor, omit all dominoes again.
The remaining planar region $\cD_{2}$ is the same as $\cD_{1}$.
Hence, it must be tiled the same way by $2\times 2$ squares.

A square on the first floor can be paired
with the square which is right above it on the second floor.
Flipping each one of these pairs creates
a tiling with all slabs vertical, or, equivalently,
identical domino-only configurations on both floors.
By the domino flip connectedness result in dimension $2$
\cite{saldanhatomei1995,thurston1990},
we can reach a given domino tiling on both floors,
proving connectivity.
\end{proof}

\smallskip
\begin{prop}
\label{prop:44N}
Consider a box of the form $4\times 4 \times N$, $N\in \NN$.
The space of slab tilings of the box
has only one flip component.
\end{prop}

\begin{proof}
This can be done by casework and essentially consists
of deducing the tree of possibilities
of the tilings for each floor.
This is similar to the computation of \textit{domino groups}
in \cite{regulardisk}.
We omit the computation, which is detailed in
\cite{notasic}.  
\end{proof}

\begin{figure}[b]
\centering
   \includegraphics[scale=0.3]{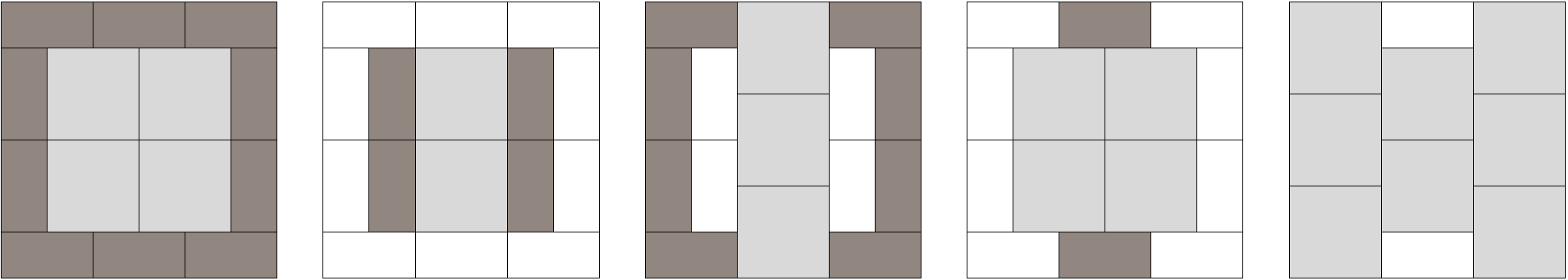}

\caption{A tiling which admits no flips;
$\Tw_\kappa=+4$ for $\kappa$ equal to red-green}
\label{665}
\end{figure}

Let us recall a fact about \textit{domino} tilings and flips:
there are often connected components of size 1,
i.e., tilings which admit no flip.
Figure~\ref{fig:cobtw} shows an example for the $3\times 3\times 2$ box;
there are similar examples in larger boxes
(see \cite{regulardisk} for examples).

An example of a box whose space of \textit{slab} tilings
has components of size one (under flips)
is the $6\times 6\times 5$ box.
One such tiling is shown in Figure~\ref{665}.
Another component (for the same box)
can be obtained by reflecting the tiling
about the plane of the paper.
This example admits slight variations for other regions.
We can add a floor comprising only slabs before the first floor,
or we can pile copies of the tiling of the above figure.
Also, we can generalize the pattern for regions
of the form $2N\times 2N\times 5$,
as in Figure~\ref{885} for the case $8\times 8\times 5$.

\begin{figure}[ht]
    \centering

    \includegraphics[scale=0.2]{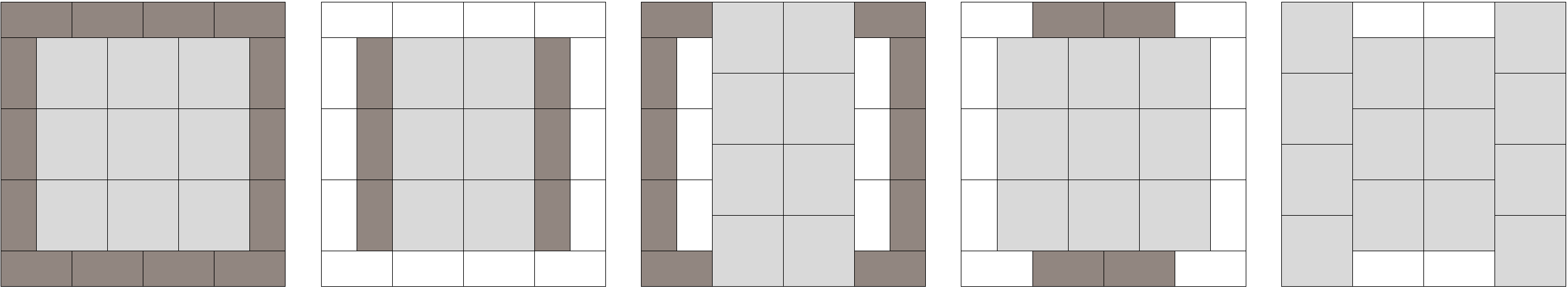}

    \caption{Another tiling which admits no flips;
    in this case, $\Tw_\kappa=+12$}
    \label{885}
\end{figure}

\bigskip

Recall that it is conjectured that the space of domino tilings of a 3d box
is connected via flips and trits.
We now show that
there is no fixed finite set of local moves
under which the set of slab tilings of arbitrary boxes becomes connected.
We first study the family of examples in
Figures~\ref{665} and \ref{885}.

\smallskip
\begin{lemma}
\label{lemma:moves}
Consider the region $\cR = [0,2k]\times [0,2k]\times [0,5l]$,
$k\geq 3$, $l \ge 1$ and a slab tiling $\bt$
following the pattern in Figure~\ref{665}
(piling copies if $l > 1$).
Now consider $\cR_0=[0,2(k-1)]\times [0,2(k-1)]\times [0,5l]$.
In $\cR$, there is only one way of removing and placing back the slabs
which are contained in the interior of $\cR_0$.
\end{lemma}

\begin{proof}
Refer to Figure~\ref{local}.
The region $\cR_0 \subset \cR$ is bounded by the red curve.
The region $\cR_1=[0,2(k-2)]\times [0,2(k-2)]\times [0,5l] \subset \cR_0$
is bounded by the dotted curve.
We proceed by induction on $k$.
We show the inductive process for five consecutive floors.
The slab numbered $1$ is clearly stuck since it is ``sandwiched"
between two other slabs.
Using this argument, on the first three floors,
we proceed to deduce the positions of slabs
which intersect the columns (in the floor diagram)
between the dotted and red lines.
In the figure, this corresponds to the slabs from $1$ to $8$.
We repeat this, but now for slabs intersecting
the rows between the dotted and red lines.
This corresponds to slabs from $9$ to $14$. 

    \begin{figure}[ht]
        \centering
   \def\svgwidth{14cm}
  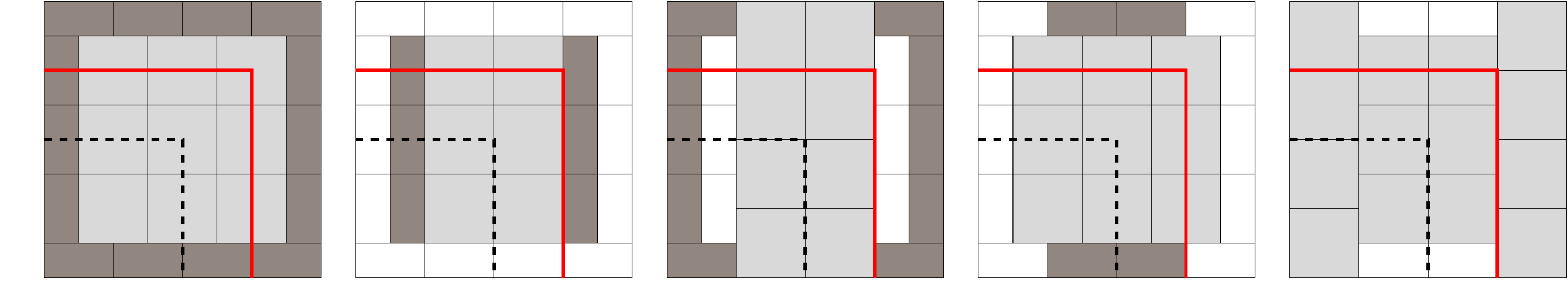
        \caption{The tiling from Figure~\ref{885}, with labels}
        \label{local}
    \end{figure}

    Now, for the last two floors, we first deduce the slabs intersecting the rows between the dotted and red lines ($15$ to $19$) and then we deduce the slabs intersecting the columns between the dotted and red lines ($20$ to $22$). We then apply the inductive step on the box bounded by the dotted curve.
\end{proof}

\bigbreak

\begin{coro}
\label{coro:nolocal}
Let $\cR$ and $\bt$ be as in Lemma~\ref{lemma:moves}.
Let $\tilde\bt \ne \bt$ be another slab tiling of $\cR$.
Then there are at least $2k - 3$ slabs in $\bt$ which are not in $\tilde\bt$.
\end{coro}

\begin{proof}
Consider tilings $\bt$ and $\tilde\bt$ as in the statement.
Let $\cX \subseteq \cR$ be the union of
slabs in $\bt$ which are not in $\tilde\bt$.
Alternatively, $\cX$ is the union of slabs in $\tilde\bt$
which are not in $\bt$. Let $d$ be the number of slabs in $\cX$.
We need to prove that $d \ge 2k - 3$.

We may assume without loss that $\cX$ has connected interior.
Indeed, otherwise let $\hat\cX \subsetneqq \cX$
be the closure of one of the connected components of the interior of $\cX$.
The region $\hat\cX$ is tiled by $\bt$ and $\tilde\bt$ in two different ways.
Let $\hat\bt$ coincide with $\tilde\bt$ in $\hat\cX$
and with $\bt$ elsewhere.
The new tiling $\hat\bt$ is a better example than $\tilde\bt$
(i.e., with a smaller value of $d$).

As in Lemma~\ref{lemma:moves} let
$\cL_{\NE} = \cR \smallsetminus \cR_0 \subset \cR$
be the L shaped region to the NE direction.
Define similar subsets $\cL_{\SE}$, $\cL_{\SW}$ and $\cL_{\NW}$.
By Lemma~\ref{lemma:moves},
we have $\cX \cap \cL_{\NE} \ne \varnothing$;
the intersections of $\cX$ with the other three L shaped sets
are likewise nonempty.
Let $\cM_{\North} = \cL_{\NE} \cap \cL_{\NW}$;
define similarly $\cM_{\South}$, $\cM_{\East}$ and $\cM_{\West}$.
We claim that at least one of the two conditions below hold:
\[ \cX \cap \cM_{\North} \ne \varnothing \ne  \cX \cap \cM_{\South},
\qquad
\cX \cap \cM_{\East} \ne \varnothing \ne  \cX \cap \cM_{\West}. \]
Indeed, assume for instance that $\cX \cap \cM_{\West} = \varnothing$:
the conditions
$\cX \cap \cL_{NW} \ne \varnothing$ and
$\cX \cap \cL_{SW} \ne \varnothing$ imply
$\cX \cap \cM_{\North} \ne \varnothing \ne  \cX \cap \cM_{\South}$;
the other cases are similar.

Assume without loss of generality that 
$\cX \cap \cM_{\North} \ne \varnothing \ne  \cX \cap \cM_{\South}$.
Consider the planes $p_x = \{x\} \times [0,2k] \times [0,5l]$,
$2 \le x \le 2k-2$,
which appear as horizontal lines in
Figures~\ref{665} and \ref{885}.
Given $x$, search for the slabs in
$\bt|_{\cX}$ or $\tilde\bt|_{\cX}$ which cross $p_x$.
We claim there are at least two such slabs.
Notice first that if there are zero such slabs
then $\cX$ becomes essentially disconnected
and we have a contradiction as above.
Suppose then that there exist exactly one such slab.
The region $\cX_{\North} = \cX \cap ([0,x] \times [0,2k] \times [0,5l])$
is tiled (by slabs, by either $\bt$ or $\tilde\bt$):
the volume $V$ of $\cX_{\North}$ is therefore a multiple of $4$.
The region minus two unit cubes is also tiled
(by either $\tilde\bt$ or $\bt$)
and therefore $V-2$ is also a multiple of $4$, a contradiction.
This completes the proof of the claim.

The total number $2d$ of slabs of both $\bt$ and $\tilde\bt$ in $\cX$
is therefore at least $2(2k-3)$, completing the proof.
\end{proof}


\begin{remark}
\label{rem:nolocal}
The corollary above is not sharp,
the best estimate appears to be $2k-1$
(for all $l \ge 1$).
The above formulation is sufficient for our main objective:
no fixed set of local moves
(involving a bounded number of slabs)
connects all slab tilings of large boxes.
More than that:
given a fixed set of local moves,
for sufficiently large $k$ the tiling $\bt$ above
admits none of these moves.
\end{remark}

\bigbreak


\section{Possible values of the triple twist}
\label{section:valuestwist}

Let $\cR$ be a region.
We denote $\TTw[\cR]=\{\TTw(\bt)|\bt \text{ is a slab tiling of } \cR\}$;
the pairs $\kappa_x, \kappa_y, \kappa_z$ are implicit.
Also, for a pair of colors $\kappa$, set
$\Tw_\kappa[\cR]=\{\Tw_\kappa(\bt)|\bt \text{ is a slab tiling of } \cR\}$.
The remainder of this section establishes bounds
on the number of values the triple twist assumes in a given region,
implying a bound on the number of slab flip connected components.
The proof is inspired by a related construction in \cite{segundoartigo}.

\bigbreak

\begin{lemma}
\label{lemma:upperconstant}
There exist constants $C_{+} > 0$ and $C_1>0$
such that the following assertions hold:
If $\cR=[0,N]^3$ is a cube with even side lengths 
then
$|\TTw(\bt)| < C_{+}N^4$ for any slab tiling $\bt$ of $\cR$.
Also for  $\cR=[0,N]^3$ we have
$\lvert\TTw[\cR]\rvert\leq C_1 N^{12}$. 
\end{lemma}

\begin{proof}
Given a pair of colors $\kappa$,
we bound $\lvert\Tw_\kappa[\cR]\rvert$.
The transformed region $\tilde{\cR}$ is contained in a cube of side $N$.
See Figures~\ref{fig:412} and \ref{bound}
for an example of the transformed region for $N=4$.
In Figure~\ref{bound},
the floor is bounded by the dotted $4\times 4$ square.
We show $3$ dominoes on the central column.
In this example, $d_y = d_z = 1$ and there is one effect.

\begin{figure}[ht]
    \centering
    \def\svgwidth{6cm}
\begingroup%
  \makeatletter%
  \providecommand\color[2][]{%
    \errmessage{(Inkscape) Color is used for the text in Inkscape, but the package 'color.sty' is not loaded}%
    \renewcommand\color[2][]{}%
  }%
  \providecommand\transparent[1]{%
    \errmessage{(Inkscape) Transparency is used (non-zero) for the text in Inkscape, but the package 'transparent.sty' is not loaded}%
    \renewcommand\transparent[1]{}%
  }%
  \providecommand\rotatebox[2]{#2}%
  \newcommand*\fsize{\dimexpr\f@size pt\relax}%
  \newcommand*\lineheight[1]{\fontsize{\fsize}{#1\fsize}\selectfont}%
  \ifx\svgwidth\undefined%
    \setlength{\unitlength}{293bp}%
    \ifx\svgscale\undefined%
      \relax%
    \else%
      \setlength{\unitlength}{\unitlength * \real{\svgscale}}%
    \fi%
  \else%
    \setlength{\unitlength}{\svgwidth}%
  \fi%
  \global\let\svgwidth\undefined%
  \global\let\svgscale\undefined%
  \makeatother%
  \begin{picture}(1,0.58703072)%
    \lineheight{1}%
    \setlength\tabcolsep{0pt}%
    \put(0,0){\includegraphics[width=\unitlength,page=1]{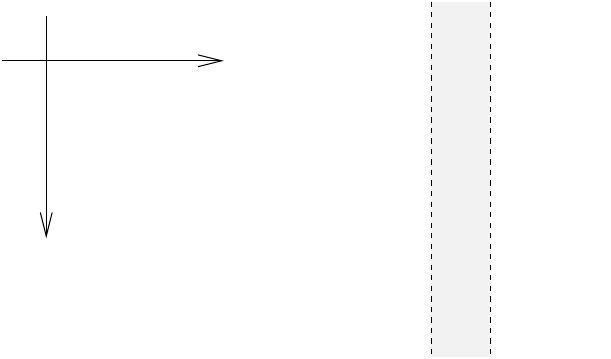}}%
    \put(0.27943247,0.40018945){\color[rgb]{0,0,0}\makebox(0,0)[lt]{\lineheight{1.25}\smash{\begin{tabular}[t]{l}$y$\end{tabular}}}}%
    \put(0.13396781,0.13835306){\color[rgb]{0,0,0}\makebox(0,0)[lt]{\lineheight{1.25}\smash{\begin{tabular}[t]{l}$x$\end{tabular}}}}%
    \put(0,0){\includegraphics[width=\unitlength,page=2]{bound6_3.pdf}}%
  \end{picture}%
\endgroup%

    \caption{A typical floor after transforming a $4\times 4\times N$ box}
    \label{bound}
\end{figure}
    
Fix $f$ a floor of $\tilde{\cR}$, fix a column $c$ of $f$,
and use axes as in Figure~\ref{bound}.
Consider dominoes with a square in $c$:
there are $d_y$ dominoes in the direction of the $y$-axis
and $d_z$ dominoes in the direction of the $z$-axis.
The number of possible effects happening within $c$ is at most $d_yd_z$.
Since $d_y+d_z\leq N$, we have $d_yd_z\leq \frac{N^2}{4}$ by AM-GM inequality.
     
Now, varying $c$ and then $f$,
we see there are at most $\frac{N^4}{4}$ effects within $\tilde{\cR}$.
Thus, for every tiling $\bt$ we have $|\Tw_\kappa(\bt)| \le \frac{N^4}{16}$.
The first claim therefore holds for $C_{+} = \frac14$.

Since $N\geq 2$,
we have 
\[ \lvert\Tw_\kappa[\cR]\rvert \leq
\frac{N^4}{8}+1\leq \frac{5}{32}N^{4}. \]
Since the above estimate holds for $\kappa_x$, $\kappa_y$ and $\kappa_z$,
we obtain
\[ \lvert\TTw[\cR]\rvert\leq \left(\frac{5}{32}\right)^3N^{12}
\leq \frac{1}{2^8}N^{12}. \]
This completes the proof, with $C_1=\frac{1}{2^8}$.
\end{proof}

\smallskip
\begin{lemma}
\label{solenoid}
Let $\cR=[0,6n]\times [0,6n] \times [0,8n]$ be a box.
For any $u,v\in [-n^2,n^2]\cap \ZZ$
we construct a slab tiling $\bt_{u,v}$ of $\cR$
such that $\TTw(\bt_{u,v})=(0,0,2uv)$.
\end{lemma}

\begin{proof}
Consider the subdivision of the $[0,6n]\times [0,6n]\times [0,8n]$ box
in parts $\alpha,\beta$ and $\gamma$ as in Figure~\ref{bound2}.
We tile $\alpha$ with horizontal slabs.
We can tile $\beta$ and $\gamma$ both as collections of $2n^2$ annuli.
In Figure~\ref{bound2}, there are four kinds of floors.
Each of them appears $2n$ times.
Each rectangle on each floor has dimensions either
$6n\times 2n$, $2n\times 6n$ or $2n\times 2n$.

\begin{figure}[ht]
    \centering
  \def\svgwidth{12cm}
    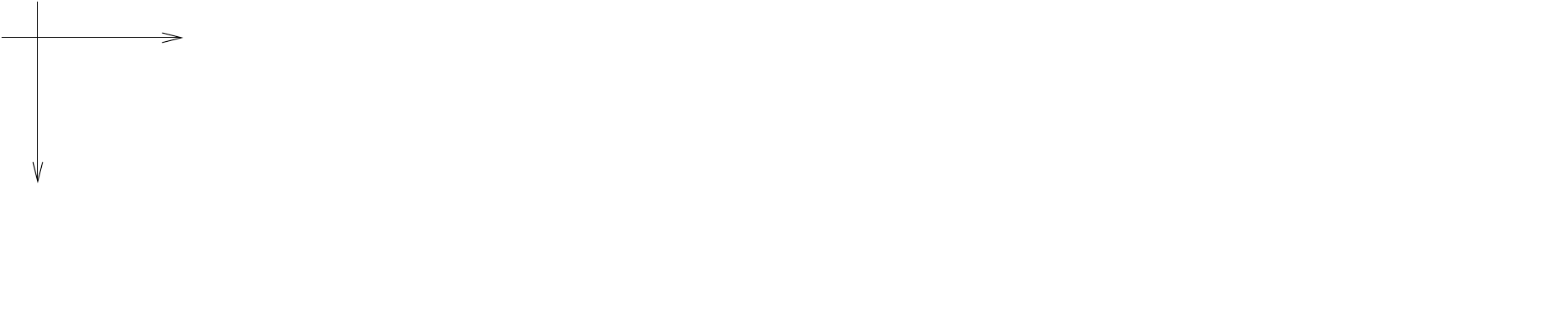

    \caption{Floor diagram of subdivision of
the $[0,6n]\times [0,6n]\times [0,8n]$ box}
    \label{bound2}
\end{figure}

Consider the surfaces $S_{\beta}$ and $S_{\gamma}$
which are both squares of side $2n$,
depicted as red cuts in Figure~\ref{bound6}.
Note that the top left vertex of a floor on the upper (resp. lower) diagram
is of the form $(0,*,0)$ (resp. $(*,0,0)$).
The surfaces $S_{\beta}$ and $S_{\gamma}$ are depicted in red.

\begin{figure}[ht]
\centering
   \def\svgwidth{10cm}
\begingroup%
  \makeatletter%
  \providecommand\color[2][]{%
    \errmessage{(Inkscape) Color is used for the text in Inkscape, but the package 'color.sty' is not loaded}%
    \renewcommand\color[2][]{}%
  }%
  \providecommand\transparent[1]{%
    \errmessage{(Inkscape) Transparency is used (non-zero) for the text in Inkscape, but the package 'transparent.sty' is not loaded}%
    \renewcommand\transparent[1]{}%
  }%
  \providecommand\rotatebox[2]{#2}%
  \newcommand*\fsize{\dimexpr\f@size pt\relax}%
  \newcommand*\lineheight[1]{\fontsize{\fsize}{#1\fsize}\selectfont}%
  \ifx\svgwidth\undefined%
    \setlength{\unitlength}{1674bp}%
    \ifx\svgscale\undefined%
      \relax%
    \else%
      \setlength{\unitlength}{\unitlength * \real{\svgscale}}%
    \fi%
  \else%
    \setlength{\unitlength}{\svgwidth}%
  \fi%
  \global\let\svgwidth\undefined%
  \global\let\svgscale\undefined%
  \makeatother%
  \begin{picture}(1,0.49940263)%
    \lineheight{1}%
    \setlength\tabcolsep{0pt}%
    \put(0,0){\includegraphics[width=\unitlength,page=1]{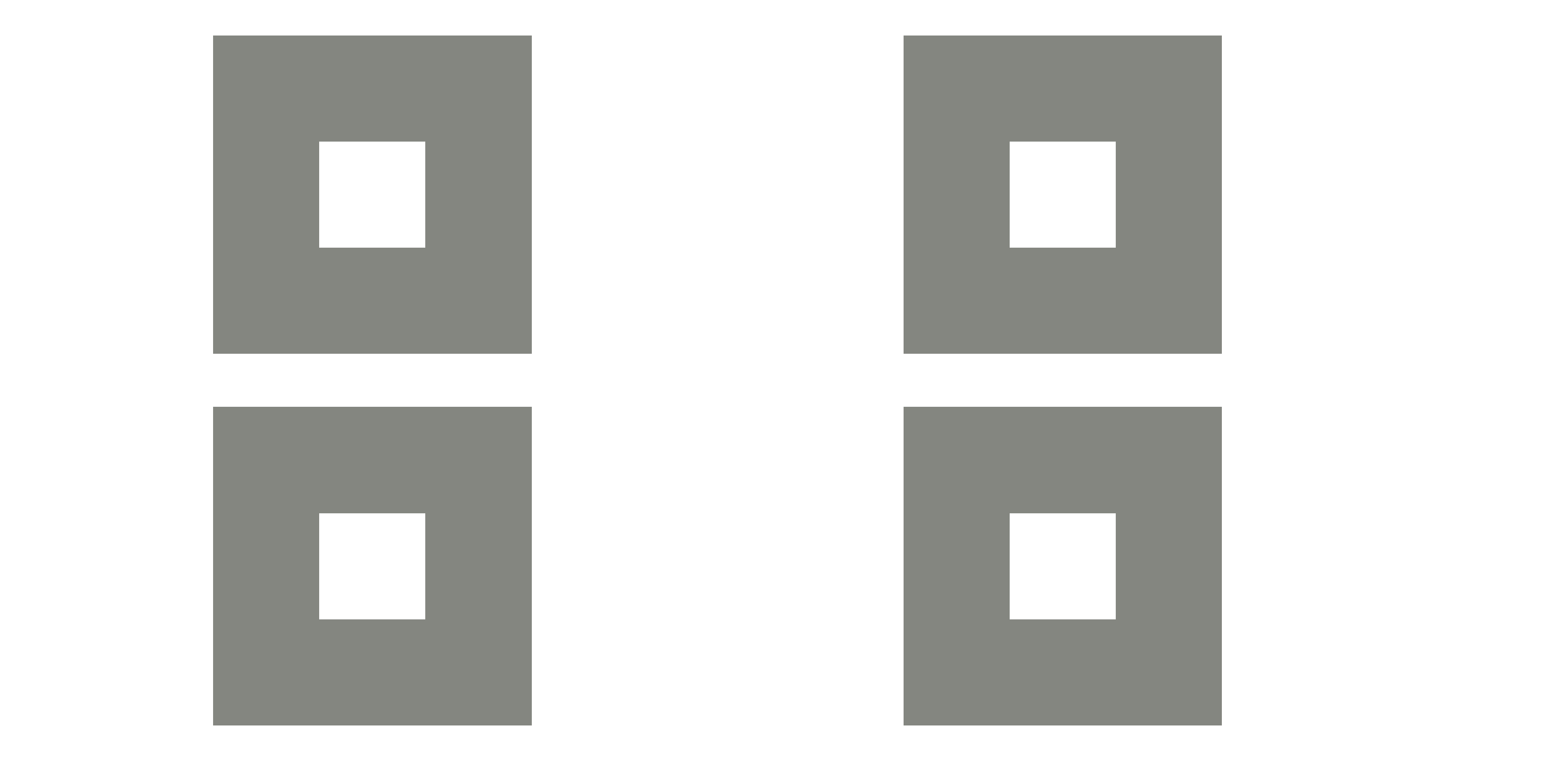}}%
    \put(0.2375254,0.00269297){\color[rgb]{0,0,0}\makebox(0,0)[lt]{\lineheight{1.25}\smash{\begin{tabular}[t]{l}$S_{\gamma}$\end{tabular}}}}%
    \put(0,0){\includegraphics[width=\unitlength,page=2]{bound6_4_2.pdf}}%
    \put(0.04312486,0.35830372){\color[rgb]{0,0,0}\makebox(0,0)[lt]{\lineheight{1.25}\smash{\begin{tabular}[t]{l}$z$\end{tabular}}}}%
    \put(0.09053963,0.44237085){\color[rgb]{0,0,0}\makebox(0,0)[lt]{\lineheight{1.25}\smash{\begin{tabular}[t]{l}$x$\end{tabular}}}}%
    \put(0,0){\includegraphics[width=\unitlength,page=3]{bound6_4_2.pdf}}%
    \put(0.2375254,0.49377448){\color[rgb]{0,0,0}\makebox(0,0)[lt]{\lineheight{1.25}\smash{\begin{tabular}[t]{l}$S_{\beta}$\end{tabular}}}}%
    \put(0.04312486,0.12122989){\color[rgb]{0,0,0}\makebox(0,0)[lt]{\lineheight{1.25}\smash{\begin{tabular}[t]{l}$y$\end{tabular}}}}%
    \put(0.09053963,0.20529702){\color[rgb]{0,0,0}\makebox(0,0)[lt]{\lineheight{1.25}\smash{\begin{tabular}[t]{l}$z$\end{tabular}}}}%
    \put(0,0){\includegraphics[width=\unitlength,page=4]{bound6_4_2.pdf}}%
  \end{picture}%
\endgroup%

\caption{Example of $\beta$ ($y$-axis view) and $\gamma$ ($x$-axis view)
on the $[0,12]\times [0,12]\times [0,16]$ box}
\label{bound6}
\end{figure}

Look at $S_{\beta}$ from the $y$-axis view.
We regard each floor of $S_{\beta}$ with only grey pieces (a grey floor)
as a domino tiling of an annulus,
and we consider the flux with respect to the red cut.
The flux across $S_{\beta}$ is the sum of the fluxes on all grey floors.
The flux across $S_{\gamma}$ is defined analogously.
We see the fluxes across each of the surfaces $S_{\beta}$ and $S_{\gamma}$
are integers in the interval $[-n^2,n^2]$. 
We can choose the fluxes across $S_{\beta}$ and $S_{\gamma}$
to be equal to $u$ and $v$, respectively.
We call this tiling $\bt=\bt_{u,v}$.

    We now verify $\bt$ satisfies the desired properties.
We will first see $\Tw_{\kappa_x}(\bt)=\Tw_{\kappa_y}(\bt)=0$.
Without loss, we prove the first case.
Consider $\bt_{\kappa_x}$ and refer to Figure~\ref{fig:verdeamarelo}.
After a sequence of flips, the slabs of $\alpha$ can be made horizontal.
Slabs from $\beta$ and $\gamma$ from floors from $(2n+1)$-th to $(4n)$-th,
that is, where $\beta$ passes through $\gamma$,
are transformed into vertical dominoes in $\tilde \bt_{\kappa_x}$.
These vertical dominoes do not affect the ones from $\alpha$.
Now, for $1\leq j\leq 2n$,
the floors $j$ and $6n+1-j$ in $\tilde \bt_{\kappa_x}$ are equal,
except for the vertical dominoes,
which have different colors on the floor diagram.
This means the effects on floors $j$ and $6n+1-j$ cancel.
Thus, $\Tw_{\kappa_x}(\bt)=0$.

    We now prove $\Tw_{\kappa_z}(\bt)=\pm 2uv$.
We will consider $\tilde{\bt}$ as a dimer covering $\bt^*$
on the graph $\cR^*$ whose vertices are centers of squares in $\tilde{\cR}$
and edges join centers of adjacent squares.
We draw $\cR^*$ on top of $\cR$ for ease of visualization.
We put red dots on the squares of $\cR$ painted with good colors.
We proceed to define surfaces
$\tilde S_{\beta}$ and $\tilde S_{\gamma}$ as follows.
From the $z$-axis view,
look at the first floor of $\beta$ which intersects $S_{\beta}$.
Consider the red dots immediately above or below the red cut
and draw a polygonal line joining them from left to right
as in Figure~\ref{bound7}.
Then, $\tilde S_{\beta}$ is a curtain with base this polygonal line
and height $2n-1$.
The definition of the surface $\tilde S_{\gamma}$ is analogous:
we now consider red dots to the left or right of the red cut,
also in Figure~\ref{bound7}.
In Figure~\ref{bound7},
the polygonal lines of $\tilde S_{\beta}$ and $\tilde S_{\gamma}$
are dotted and the surfaces $S_{\beta}$ and $S_{\gamma}$
are represented as full lines.

    \begin{figure}[ht]
        \centering
    	\def\svgwidth{10cm}
\begingroup%
  \makeatletter%
  \providecommand\color[2][]{%
    \errmessage{(Inkscape) Color is used for the text in Inkscape, but the package 'color.sty' is not loaded}%
    \renewcommand\color[2][]{}%
  }%
  \providecommand\transparent[1]{%
    \errmessage{(Inkscape) Transparency is used (non-zero) for the text in Inkscape, but the package 'transparent.sty' is not loaded}%
    \renewcommand\transparent[1]{}%
  }%
  \providecommand\rotatebox[2]{#2}%
  \newcommand*\fsize{\dimexpr\f@size pt\relax}%
  \newcommand*\lineheight[1]{\fontsize{\fsize}{#1\fsize}\selectfont}%
  \ifx\svgwidth\undefined%
    \setlength{\unitlength}{712bp}%
    \ifx\svgscale\undefined%
      \relax%
    \else%
      \setlength{\unitlength}{\unitlength * \real{\svgscale}}%
    \fi%
  \else%
    \setlength{\unitlength}{\svgwidth}%
  \fi%
  \global\let\svgwidth\undefined%
  \global\let\svgscale\undefined%
  \makeatother%
  \begin{picture}(1,0.48174157)%
    \lineheight{1}%
    \setlength\tabcolsep{0pt}%
    \put(0,0){\includegraphics[width=\unitlength,page=1]{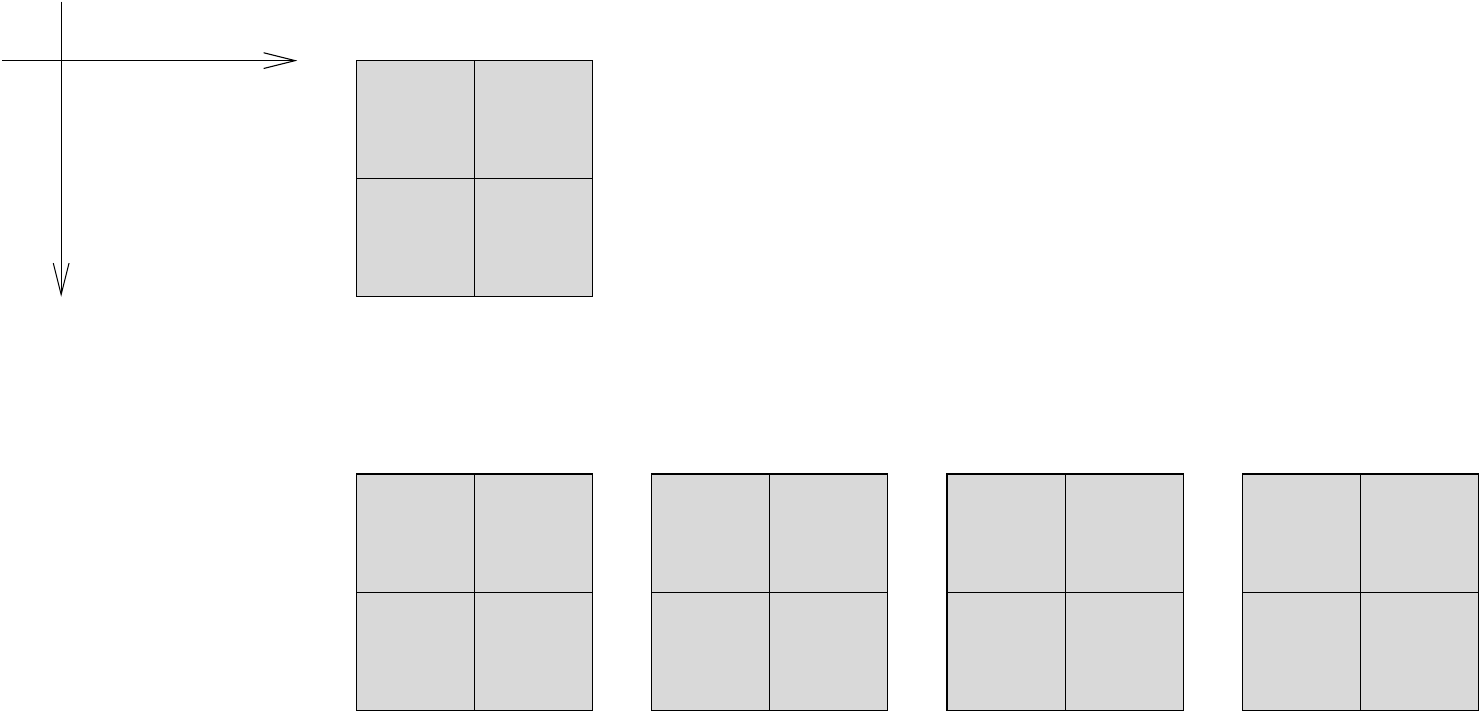}}%
    \put(0.07306812,0.24976243){\color[rgb]{0,0,0}\makebox(0,0)[lt]{\lineheight{1.25}\smash{\begin{tabular}[t]{l}$x$\end{tabular}}}}%
    \put(0.15269483,0.3930905){\color[rgb]{0,0,0}\makebox(0,0)[lt]{\lineheight{1.25}\smash{\begin{tabular}[t]{l}$y$\end{tabular}}}}%
    \put(0,0){\includegraphics[width=\unitlength,page=2]{bound6_4_3.pdf}}%
  \end{picture}%
\endgroup%

        \caption{Construction of the surfaces
$\tilde S_{\beta}$ (above) and $\tilde S_{\gamma}$ (below)
following the example in Figure~\ref{bound6}}
        \label{bound7}
    \end{figure}
    
    We see that if we twist one annulus of $\beta$ (resp $\gamma$)
the flux across $S_{\beta}$ (resp $S_{\gamma}$)
of the domino tiling in Figure \ref{bound6} changes by $\pm 1$.
Thus, the flux across $\tilde S_{\beta}$ (resp $\tilde S_{\gamma}$)
changes by $\pm 2$.
Note that translating the new surface cuts
does not change the respective fluxes.
We do not need to be careful with orientations and signs here.
Also, $\Tw_{\kappa_z}(\bt)$ changes by $\pm 2v$
(resp $\pm 2u$).

To prove this, without loss,
let $\bt_0$ be $\bt$ after twisting one of the annuli of $\beta$.
The difference $\tilde \bt_0-\tilde{\bt}$
has one non trivial cycle which bounds an oriented surface $\tilde{S}$.
We draw $\tilde \bt_0 - \tilde{\bt}$ on top of $\cR$
as in Figure~\ref{bound8}.
The surface $\tilde{S}$ bounded by $\tilde \bt_0-\tilde{\bt}$
is drawn as black polygonal lines.
Full lines represent the first and last floors of this cycle.
The first and last floors of two annuli from $\tilde{\gamma}$
are shown as purple polygonal lines.
Full segments represent the actual dimers in the tiling.
Notice $\tilde{S}$ is tangent to $\tilde S_{\gamma}$.
To calculate $t = \Tw(\tilde \bt_0)-\Tw(\tilde{\bt})$,
we must count dominoes from $\tilde{\bt}$ passing through $S$.
These are exactly the ones which are part of an annulus in $\tilde{\gamma}$.
By twisting one annulus from $\tilde{\gamma}$ we see $t$ changes by $\pm 2$,
hence $t= \pm 2v$ by induction.

\begin{figure}[ht]
        \centering
    \includegraphics[scale=0.2]{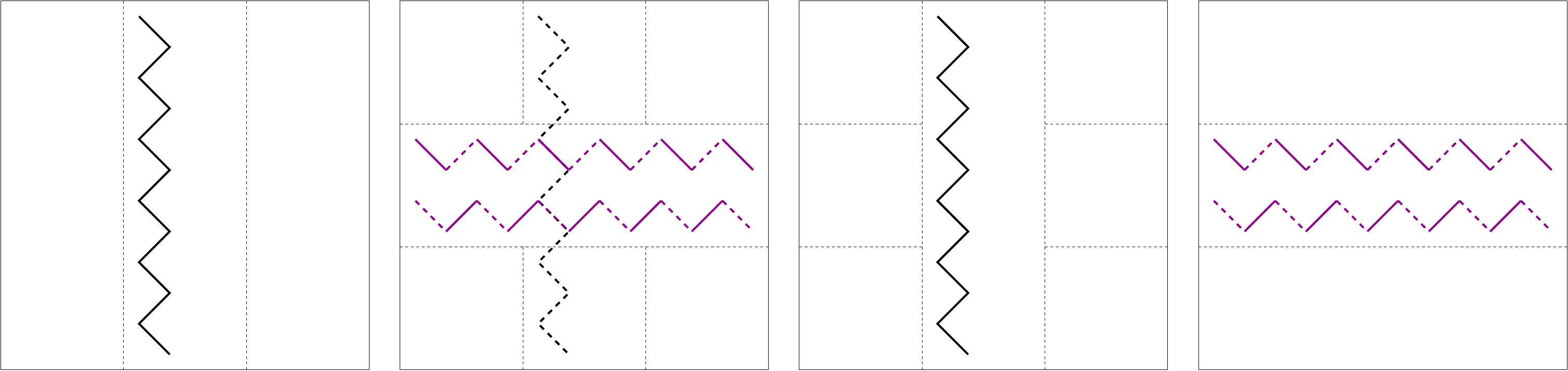}

        \caption{Some floors of the $[0,12]\times [0,12]\times [0,16]$ box}
        \label{bound8}
    \end{figure}

\begin{figure}[ht]
        \centering
  \def\svgwidth{14cm}
\begingroup%
  \makeatletter%
  \providecommand\color[2][]{%
    \errmessage{(Inkscape) Color is used for the text in Inkscape, but the package 'color.sty' is not loaded}%
    \renewcommand\color[2][]{}%
  }%
  \providecommand\transparent[1]{%
    \errmessage{(Inkscape) Transparency is used (non-zero) for the text in Inkscape, but the package 'transparent.sty' is not loaded}%
    \renewcommand\transparent[1]{}%
  }%
  \providecommand\rotatebox[2]{#2}%
  \newcommand*\fsize{\dimexpr\f@size pt\relax}%
  \newcommand*\lineheight[1]{\fontsize{\fsize}{#1\fsize}\selectfont}%
  \ifx\svgwidth\undefined%
    \setlength{\unitlength}{3007bp}%
    \ifx\svgscale\undefined%
      \relax%
    \else%
      \setlength{\unitlength}{\unitlength * \real{\svgscale}}%
    \fi%
  \else%
    \setlength{\unitlength}{\svgwidth}%
  \fi%
  \global\let\svgwidth\undefined%
  \global\let\svgscale\undefined%
  \makeatother%
  \begin{picture}(1,0.29697373)%
    \lineheight{1}%
    \setlength\tabcolsep{0pt}%
    \put(0,0){\includegraphics[width=\unitlength,page=1]{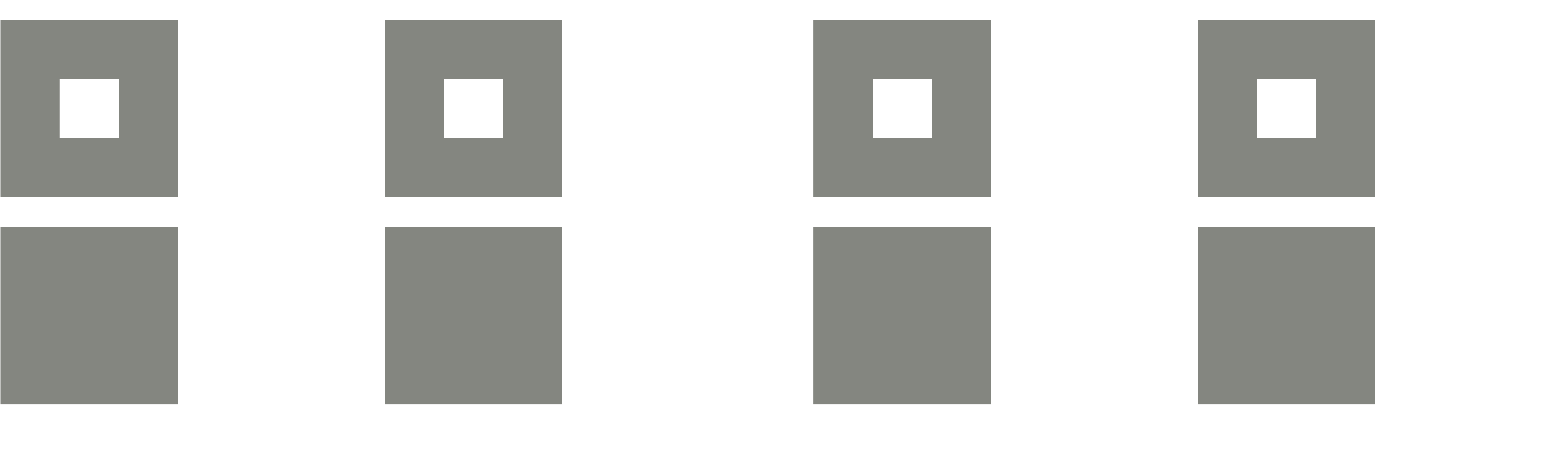}}%
    \put(0.05682675,0.00151679){\color[rgb]{0,0,0}\makebox(0,0)[lt]{\lineheight{1.25}\smash{\begin{tabular}[t]{l}$S_{\gamma}$\end{tabular}}}}%
    \put(0,0){\includegraphics[width=\unitlength,page=2]{bound6_4_5.pdf}}%
    \put(0.05682675,0.29382255){\color[rgb]{0,0,0}\makebox(0,0)[lt]{\lineheight{1.25}\smash{\begin{tabular}[t]{l}$S_{\beta}$\end{tabular}}}}%
    \put(0,0){\includegraphics[width=\unitlength,page=3]{bound6_4_5.pdf}}%
  \end{picture}%
\endgroup%

        \caption{We continue with the example of Figure~\ref{bound6}}
        \label{bound6(copy1)}
    \end{figure}
    
Finally, we claim that if either flux across $\tilde{S_{\beta}}$
or $\tilde S_{\gamma}$ is equal to $0$,
we can take $\bt$ to a tiling with only horizontal slabs,
and thus with no twist.
This means by twisting annuli of $\bt$ to end in the above zero flux situation
we end up proving $\Tw_{\kappa_z}(\bt)=\pm 2uv$.
    
Without loss of generality we treat the first case of the above claim.
Refer to Figure~\ref{bound6(copy1)}:
on the first row we make the slabs in $\beta$ horizontal,
on the second row we make the slabs in $\gamma$ and its hole horizontal.
If this holds, then we can flip the slabs in $\beta$
so that all of them are horizontal.
In particular, the slabs passing through $\gamma$'s hole are horizontal.
Look at $\gamma$ from the $x$-axis view.
For every pair of consecutive floors of $\gamma$
we can make all slabs horizontal
by using the domino flip connectedness result for planar regions
\cite{saldanhatomei1995,thurston1990}.
\end{proof}

\smallskip
\begin{lemma}
\label{lemma:16n}
Let $\cR = [0,16n]^3$ be a box.
For any triple $t\in ([-2n^4,2n^4]\cap 2\ZZ)^3$
we construct a slab tiling $\bt$ of $\cR$ such that
$\TTw(\bt) = t$.
\end{lemma}

\begin{proof}
    For numbers $u_{i,j},v_{i,j}\in [-n^2,n^2]$, $1\leq i\leq 2$, $1\leq j\leq 3$ we can construct tilings $\bt_{i,j}$ of the $[0,6n]\times [0,6n]\times [0,8n]$ box such that $\TTw(\bt_j)=2u_{i,j}v_{i,j}\mathbf{e_j}$ by Lemma \ref{solenoid}
(here $\mathbf{e_j} \in \ZZ^3$ is the $j$-th vector in the canonical base).
After doing this, we divide the $[0,16n]^3$ box in $8$ cubes of side $8n$. Each of these smaller cubes is tiled using one of the $\bt_{i,j}$ on one corner, and horizontal slabs to fill the rest of the space. Two of the smaller cubes will receive only horizontal slabs. This new tiling is named $\bt$.

Let $t=(t_1,t_2,t_3)$.
Choose 
\[ u_{1,j}=n^2, \quad v_{1,j}=\left\lfloor \frac{t_j}{2n^2}\right\rfloor,
\quad u_{2,j}=1, \quad v_{2,j}=n^2\left\{\frac{t_j}{2n^2}\right\}. \]
Here, for $s \in \RR$,
we have $s = \lfloor s \rfloor + \{s\}$
where $\lfloor s \rfloor \in \ZZ$ and $\{s\} \in [0,1)$.
We have $t_j = 2u_{1,j}v_{1,j} + 2u_{2,j}v_{2,j}$,
as desired.
Then, for each $j$,
$\TTw(\bt_{1,j}+\bt_{2,j})=t_j\mathbf{e_j}\in ([-2n^4,2n^4]\cap 2\ZZ)^3$,
where $\bt_{1,j}+\bt_{2,j}$
is the tiling of the two cubes containing $\bt_{1,j}$ and $\bt_{2,j}$.
    Summing the above for all $1\leq j\leq 3$ yields $\TTw(\bt)=(t_1,t_2,t_3)$.
\end{proof}

\smallskip
\begin{lemma}
\label{lemma:lowerconstant}
    There is a constant $C_0>0$ such that the following assertion holds:
If $\cR=[0,N]^3$, $N\geq 256$, is a cube with even side lengths
(and thus slab tileable) then $\lvert\TTw[\cR]\rvert\geq C_0N^{12}$.
\end{lemma}
\begin{proof}
      Write $N=16M+2k$ for some $0\leq k\leq 7$.
We can divide $\cR$ into two parts $\cR_1$ and $\cR_2$
as in Figure~\ref{bound4}
(there is one kind of floor:
it comprises one $16M\times 16M$ square on the lower left corner,
$\cR_1$, and its complement, $\cR_2$).
For a given $\mathbf{v}\in ([-2M^4,2M^4]\cap 2\ZZ)^3$,
we tile $\cR_1$ with a tiling $\bt_1$ having $\TTw(\bt_1)=\mathbf{v}$
and we tile $\cR_2$ with horizontal slabs.
     
     \begin{figure}[ht]
         \centering
  \def\svgwidth{4cm}
\begingroup%
  \makeatletter%
  \providecommand\color[2][]{%
    \errmessage{(Inkscape) Color is used for the text in Inkscape, but the package 'color.sty' is not loaded}%
    \renewcommand\color[2][]{}%
  }%
  \providecommand\transparent[1]{%
    \errmessage{(Inkscape) Transparency is used (non-zero) for the text in Inkscape, but the package 'transparent.sty' is not loaded}%
    \renewcommand\transparent[1]{}%
  }%
  \providecommand\rotatebox[2]{#2}%
  \newcommand*\fsize{\dimexpr\f@size pt\relax}%
  \newcommand*\lineheight[1]{\fontsize{\fsize}{#1\fsize}\selectfont}%
  \ifx\svgwidth\undefined%
    \setlength{\unitlength}{367bp}%
    \ifx\svgscale\undefined%
      \relax%
    \else%
      \setlength{\unitlength}{\unitlength * \real{\svgscale}}%
    \fi%
  \else%
    \setlength{\unitlength}{\svgwidth}%
  \fi%
  \global\let\svgwidth\undefined%
  \global\let\svgscale\undefined%
  \makeatother%
  \begin{picture}(1,0.93188011)%
    \lineheight{1}%
    \setlength\tabcolsep{0pt}%
    \put(0,0){\includegraphics[width=\unitlength,page=1]{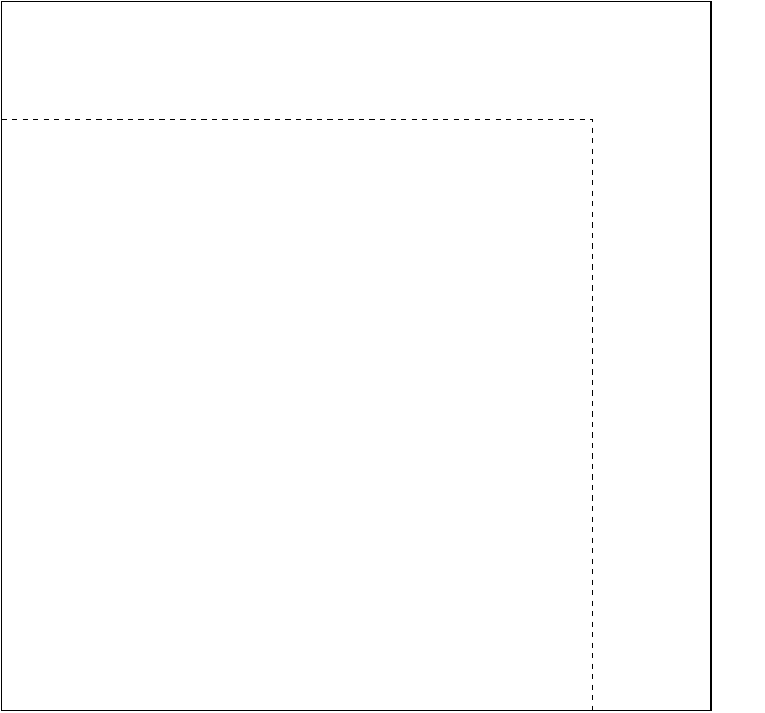}}%
    \put(0.3887551,0.38862699){\color[rgb]{0,0,0}\makebox(0,0)[lt]{\lineheight{1.25}\smash{\begin{tabular}[t]{l}$\mathcal R_1$\end{tabular}}}}%
    \put(0.77532042,0.85250537){\color[rgb]{0,0,0}\makebox(0,0)[lt]{\lineheight{1.25}\smash{\begin{tabular}[t]{l}$\mathcal R_2$\end{tabular}}}}%
  \end{picture}%
\endgroup%

         \caption{Diagram of subdivision of $\cR$}
         \label{bound4}
     \end{figure}
     
     Then,
$\lvert\TTw[\cR]\rvert\geq (2M^4)^3=8M^{12}=
8(\frac{N-2k}{16})^{12}=c(N-2k)^{12}$, where $c=\frac{1}{2^{45}}$.
We seek $C>0$ such that $(N-2k)^{12}\geq CN^{12}$ holds.
This happens if and only if $C\leq (1-\frac{2k}{N})^{12}$.
By minimizing the RHS with the constraints $N\geq 256$ and $k\leq 7$,
we can take $C=\frac{1}{2}\leq (1-\frac{14}{256})^{12}$.
Finally, we ca $\cR=[0,N]^3$n take $C_0=cC=\frac{1}{2^{46}}$.
\end{proof}

We summarize the previous lemmas in the proposition below.
\smallskip
\begin{prop}
\label{prop:bound}
For $C_1 = \frac{1}{2^8}$ and $C_0=\frac{1}{2^{46}}$,
the following assertion holds: \\
If $N\geq 256$ is an even integer,
then $$C_0N^{12}\leq \lvert \TTw[[0,N]^3]\rvert\leq C_1N^{12}.$$
\end{prop}
\begin{proof}
We take $C_0$  and $C_1$ as in the proofs of
Lemmas \ref{lemma:upperconstant} and \ref{lemma:lowerconstant}.
\end{proof}

\bigbreak

We are also ready for
Theorem~\ref{theo:tripletwist}.

\begin{proof}[Proof of Theorem~\ref{theo:tripletwist}]
We use the max norm and give explicit values for $C_{\pm}$
assuming $N \ge 256$.
As in Lemma~\ref{lemma:upperconstant}, we can take $C_{+} = \frac14$.
We claim that $C_{-} = 1/{2^{16}}$ works.
Indeed, take $n = \left\lfloor \frac{N}{16} \right\rfloor$
and consider the smaller cube $\cR_{-} = [0,16n]^3$.
Notice that $N^4 < 2^{17} n^4$.
Each coordinate of $t$ is an even integer
with absolute value smaller than $C_{-} N^4 < 2n^4$.
Apply Lemma~\ref{lemma:16n} to deduce that
there exists a tiling $\bt_{-}$ of $\cR_{-}$
with $\TTw(\bt_{-}) = t$.
Finally, construct a tiling $\bt$ of $\cR = [0,N]^3 \supseteq [0,16n]^3$
by starting with $\bt_{-}$ and completing the remaining region
with horizontal slabs.
We have $\TTw(\bt) = \TTw(\bt_{-}) = t$, as desired.
\end{proof}


\section{Final remarks}
\label{section:final}

There are many aspects of the theory of slab tilings
that we seek to understand better.
We list a few.

The first is that in proposition \ref{prop:bound},
the bounding constants were not optimised and thus are far from being sharp.

The second is that we would like to study the flip connected components in more classes of examples, not just the ones in proposition \ref{prop:44N}. For example, we do not know what happens in boxes of height $3$, but we guess there is only one component.

The third is regarding probabilistic results such as the ones for domino tilings in \cite{normal}. For example, for slab tilings we would like to know the distribution for the sizes of connected components of large regions. 

The fourth is regarding other moves besides the flip.
In the domino case, there is also the trit,
and the question on whether all tilings of a given region
are connected by flips and trits is open.
In the slab case,
the existence of a set of local moves
which connect all or most tilings of a given region
is hindered by Remark~\ref{rem:nolocal}.
In the example of Figure~\ref{885}, however,
it is possible to perform moves involving a quantity of pieces
proportional to the lengths of the sides of the box.
It can also be asked how common is a situation comparable
to the example in Lemma~\ref{lemma:moves}
and Corollary~\ref{coro:nolocal}.
It would be interesting to clarify the situation in greater generality.

Also, one property observed in the examples of slab tilings
we studied so far concerns the parity of the entries of the triple twist.
Our final and main conjecture is the following.

\smallskip
\begin{conj}
Let $\cR$ be a region tileable by slabs.
Then there exist $t_x, t_y, t_z
\in \{0,1\} \subset \ZZ$
with the following property:
For every slab tiling $\bt$ of $\cR$, the integers
\[ \Tw_{\kappa_x}(\bt) - t_x, \quad
\Tw_{\kappa_y}(\bt) - t_y, \quad \Tw_{\kappa_z}(\bt) - t_z \]
are all even. If $\cR$ is a box then $t_x = t_y = t_z = 0$.
\end{conj}

\begin{figure}[ht]
        \centering
    \includegraphics[scale=0.5]{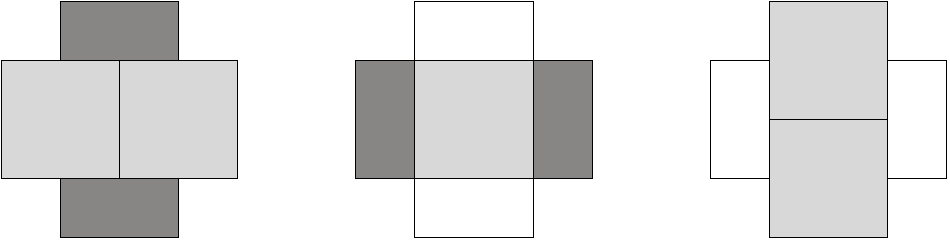}

        \caption{A tiling $\bt$ with $\Tw_{\kappa_z}(\bt) = \pm 1$}
        \label{odd}
    \end{figure}

Figure~\ref{odd} shows a region for which
the conjecture holds with $t_z = 1$.

\section{Declarations}

\subsection{Funding and Acknowledgements}
The second author is thankful for the generous support of
CNPq, CAPES and FAPERJ (Brazil).
The first and third authors are thankful to Projeto Arquimedes and to CNPq
for the support in the PIBIC undergraduate project.




\end{document}